\documentclass[reqno]{amsart}
\usepackage{tikz}
\usepackage{xcolor}
\usepackage{amssymb,latexsym,amsmath,extarrows}
\usepackage{amsthm}
\usepackage{graphicx,mathrsfs,comment}
\usepackage{hyperref,url}
\usepackage{pict2e}
\usepackage{enumerate}
\usepackage{bm}
\usepackage{microtype}
\usepackage{cancel}

\usepackage{amstext}
\usepackage{bbm} 

\numberwithin{equation}{section}

\newcommand{\orcid}[1]{\href{https://orcid.org/#1}{\texttt{ORCID: #1}}}

\usepackage{esint}

\newtheorem{theorem}{Theorem}[section]
\newtheorem{lemma}[theorem]{Lemma}

\newtheorem{proposition}[theorem]{Proposition}

\newtheorem*{remark*}{Remark}

\newtheorem{definition}[theorem]{Definition}
\newtheorem{corollary}[theorem]{Corollary}

\newtheorem{conjecture}[theorem]{Conjecture}

\makeatletter
\newcommand{\barredsum}{%
  \DOTSB\mathop{\mathpalette\@barredsum\relax}\slimits@
}
\newcommand{\@barredsum}[2]{%
  \begingroup
  \sbox\z@{$#1\sum$}%
  \setlength{\unitlength}{\dimexpr2pt+\ht\z@+\dp\z@\relax}%
  \@barredsumthickness{#1}%
  \vphantom{\@barredsumbar}%
  \ooalign{$\m@th#1\sum$\cr\hidewidth$#1\@barredsumbar$\hidewidth\cr}%
  \endgroup
}
\newcommand{\@barredsumbar}{%
  \vcenter{\hbox{\begin{picture}(0,1)\roundcap\Line(0,0)(0,1)\end{picture}}}%
}
\newcommand{\@barredsumthickness}[1]{
  \linethickness{%
    1.25\fontdimen8
      \ifx#1\displaystyle\textfont\else
      \ifx#1\textstyle\textfont\else
      \ifx#1\scriptstyle\scriptfont\else
      \scriptscriptfont\fi\fi\fi 3
  }%
}
\makeatother

\begin{document}

\title{On Cone Restriction Estimates in Higher Dimensions}

\date{}

\author{Xiangyu Wang} \address{Xiangyu Wang \\ Department of Mathematics\\ University of Illinois Urbana-Champaign\\ Urbana, IL 61801, USA\\ \orcid{0009-0003-5983-6961}} \email{xw70@illinois.edu}


\begin{abstract}

We revisit the approach of Ou and Wang to the cone restriction problem via polynomial partitioning. By recasting their inductive scheme as a recursive algorithm and incorporating the nested polynomial Wolff axioms, we obtain improved bounds for cone restriction estimates in higher dimensions.

\end{abstract}
\maketitle


\section{Introduction}\label{intro}
\subsection{Statement of results}

We obtain improved cone restriction estimates in higher dimensions. The Fourier restriction problem is one of the central open questions in harmonic analysis. It concerns the possibility of restricting the Fourier transform of an $L^p$ function to a curved hypersurface. Stein \cite{S} conjectured that, for hypersurfaces such as the sphere, the paraboloid, and the cone, the Fourier transform admits a bounded restriction operator from $L^{p'}$ to $L^{q'}$ for an appropriate range of exponents $p'$ and $q'$.

We now state the cone restriction conjecture precisely.

\begin{conjecture}\label{cone restriction}
Let $n \geq 3$, and let $\mathcal{C}$ be the truncated cone in $\mathbb{R}^n$ defined by \begin{equation}\label{cone}
\mathcal{C} := \{ (\xi,|\xi|) : \xi\in\mathbb{R}^{n-1},\ 1\leq |\xi|\leq 2\}.
\end{equation}
Then for all $p, q$ satisfying 
\begin{equation}\label{fullrange}
p > \frac{2(n-1)}{n-2}, \frac{n-2}{q'} \geq \frac{n}{p}
\end{equation}
we have $$
\| \hat{f} \|_{L^{q'}(\mathcal{C};d\sigma)} \lesssim_{p, q} \|f\|_{L^{p'}(\mathbb{R}^n)}, \forall f: \mathbb{R}^n \to \mathbb{C}
$$ where $d\sigma$ denotes the induced surface measure on $\mathcal{C}$, and $p',q'$ are the Hölder conjugates of $p,q$. 
\end{conjecture} 

Define the Fourier extension operator associated with $\mathcal{C}$ by \[
Ef(x) := \int_{2\overline{B}^{n-1} \setminus B^{n-1}} e^{i(x'\cdot\xi+x_n|\xi|)} f(\xi)\,d\xi,
\] where $B^{n-1}\subset\mathbb{R}^{n-1}$ is the unit ball centered at the origin. The parametrization $\xi\mapsto(\xi,|\xi|)$ identifies surface measure with a constant multiple of $d\xi$. By duality, Conjecture~\ref{cone restriction} is equivalent to the following statement. 
\begin{conjecture}\label{extention operator}
Let $n \geq 3$. For all $p,q$ satisfying \eqref{fullrange}, the estimate
\begin{equation}\label{linearestimate0}
    \| Ef \|_{L^p(\mathbb{R}^n)} \lesssim_{p, q} \|f\|_{L^q(\mathbb{R}^{n-1})}
\end{equation}
holds for all $f\in L^q(\mathbb{R}^{n-1})$.
\end{conjecture}
Conjecture~\ref{extention operator} is known in dimensions $n=3,4,5$ through the work of Barcel\'o \cite{B}, Wolff \cite{Wol1}, and Ou--Wang \cite{OW}. Our starting point in higher dimensions is the following theorem of Ou and Wang.
\begin{theorem}[\cite{OW}, Theorem 2] Let $n \geq 3$. Suppose the tuple $(p, q, k)$ is admissible in the sense that $2<q\leq p$, $2 \leq k \leq n$ and $$
\left\{ \begin{array}{rcl}
 p > 2\frac{n+2}{n}, q' \leq \frac{n-2}{n}p & \mathrm{if}
 & k=2 \\ p > 2\frac{n+k}{n+k-2}, p \geq \frac{n}{\frac{2n-k-1}{2}-\frac{n-k+1}{q}} & \mathrm{if} & k \geq 3
 \end{array}\right.
$$ Then the estimate \eqref{linearestimate0} holds for all $f\in L^q(\mathbb{R}^{n-1})$.
\end{theorem}
This theorem yields the following corollary.
\begin{corollary}[\cite{OW}, Theorem 1] \label{OW} Let $n \geq 3$. For all $p$ satisfying
\begin{equation} \label{pre-result}
    p >  \left\{ \begin{array}{rcl}
 4 & \ \mathrm{if}
 & n=3 \\ 2\frac{3n+1}{3n-3} & \mathrm{if} & n >3 \quad \mathrm{odd} \\
 2\frac{3n}{3n-4} & \mathrm{if} & n > 3 \quad \mathrm{even}
 \end{array}\right.
\end{equation}
the estimate
\begin{equation}\label{lp2lp}
    \| Ef \|_{L^p(\mathbb{R}^n)} \lesssim_p \|f\|_{L^p(\mathbb{R}^{n-1})}
\end{equation}
holds for all $f\in L^p(\mathbb{R}^{n-1})$.
\end{corollary} 
We now state our main result. Given integers $n \geq 4$ and $k$ with $2 \leq k \leq n - 2$, define $p(n,k)$ and $q(n,k)$ by
\[p(n,k) := 2+\frac{12}{4n-5+2\left(k-1\right)\prod^{n-2}_{i=k}{\frac{2i}{2i+1}}}\]
\[\frac1{q(n,k)} = \frac1{p(n,k)} + \frac{1}{4n-8}\big(\frac6{p(n,k)} - 1\big)\]
For $q \geq 2$, we define $p_{n,k}(q)$ by
\begin{equation*}
p_{n,k}(q) :=
\begin{cases}
p(n,k)
& \text{if $q \geq q(n,k)$},\\[4pt]
\big (\frac{n+k-2}{2n+2k}\alpha_{n,k}(q) + \frac{1-\alpha_{n,k}(q)}{p(n,k)}\big )^{-1}
& \text{if $2 \leq q < q(n,k)$}.\\[4pt]
\end{cases}
\end{equation*}
where $\alpha_{n,k}(q)$ is defined by
\[\frac{1}{q} = \frac{\alpha_{n,k}(q)}{2} + \frac{1-\alpha_{n,k}(q)}{q(n,k)}\]
\begin{theorem} \label{mainresult} Let $n \geq 4$. Suppose the tuple $(p, q, k)$ is admissible in the sense that $2\leq q \leq p$, $2 \leq k \leq n-2$ and
\[\begin{cases}
p > p_{n,2}(q), q' \leq \frac{n-2}{n}p
& \text{if $k =2$},\\[4pt]
p > p_{n,k}(q), p \geq \frac{n}{\frac{2n-k-1}{2}-\frac{n-k+1}{q}}
& \text{if $k \geq 3$}.\\[4pt]
\end{cases}\]
then the estimate \eqref{linearestimate0} holds for all $f\in L^q(\mathbb{R}^{n-1})$.
\end{theorem}
Arguing as in \cite{HZ}, we obtain
\begin{corollary} \label{lplpresult} As $n\to\infty$, the estimate \eqref{lp2lp} holds for all $f\in L^p(\mathbb{R}^{n-1})$ in the range
\begin{equation}\label{currentrange}
p > 2 + \lambda n^{-1} + O(n^{-2}),\quad \lambda  = 2.596\ldots
\end{equation}
\end{corollary}
Observe that \eqref{pre-result} implies that the estimate \eqref{lp2lp} holds for all 
$f:\mathbb{R}^{n-1}\to\mathbb{C}$ whenever
\[p>2+\frac{8}{3n}+O(n^{-2}).\]
Consequently, Corollary \ref{lplpresult} yields an improved $L^p\to L^p$ bound for the cone restriction conjecture in higher dimensions.

Let $B_R \subset \mathbb{R}^n$ denote a ball of radius $R$. Restricted-type real interpolation, a standard $\epsilon$-removal argument (see \cite{GHI,T,XW}), and bilinear interpolation at the boundary (see \cite{OW}) reduce Theorem~\ref{mainresult} to the following local estimate. 
\begin{proposition} \label{result} Let $n \geq 4$. Suppose the tuple $(p, q, k)$ satisfies $2\leq q \leq p$, $2 \leq k \leq n-2$, and
\[\begin{cases}
p \ge p_{n,2}(q), q' \leq \frac{n-2}{n}p
& \text{if $k =2$},\\[4pt]
p \ge p_{n,k}(q), p \geq \frac{n}{\frac{2n-k-1}{2}-\frac{n-k+1}{q}}
& \text{if $k \geq 3$}.\\[4pt]
\end{cases}\]
then the estimate
\begin{equation} \label{mixed-norm-result}
\| Ef \|_{L^p(B_R)} \lesssim_{p, q, \varepsilon} R^\varepsilon \|f\|_{L^2(\mathbb{R}^{n-1})}^\frac{2}{q}\|f\|_{L^\infty(\mathbb{R}^{n-1})}^{1-\frac{2}{q}}
\end{equation}
holds uniformly for all Schwartz functions $f$ with $\operatorname{supp} f \subset 2\overline{B}^{\,n-1} \setminus B^{n-1}$ and $R \geq 1$.  
\end{proposition}

Fourier restriction is connected to several central problems in harmonic analysis, geometric measure theory, and partial differential equations. For hypersurfaces with nonvanishing Gaussian curvature, such as the sphere and the paraboloid, these connections include the Bochner--Riesz problem, Kakeya-type estimates, and the Schr\"odinger maximal problem. The cone has a different geometry because one of its principal curvatures vanishes. Nevertheless, restriction estimates for these surfaces are closely related: the restriction conjecture for the sphere or paraboloid in $\mathbb{R}^n$ implies the corresponding conjecture for the cone in $\mathbb{R}^{n+1}$; see \cite{N}.

Guth \cite{G1,G2} introduced \emph{polynomial partitioning} into Fourier restriction theory and obtained improved estimates for the paraboloid. Subsequent developments include the work of Wang \cite{Wa}, Wang--Wu \cite{WW}, Hickman--Rogers \cite{HR}, Hickman--Zahl \cite{HZ}, and Guo--Wang--Zhang \cite{GWZ}. In \cite{HR,HZ}, Guth's induction-on-scales argument was reformulated as a recursive algorithm. This formulation permits the use of Kakeya-type geometric estimates from \cite{GZ,KR,Z}, including the \emph{nested polynomial Wolff axioms}, to control the number of tube directions compatible with a sequence of algebraic varieties. Polynomial partitioning has also been applied to the Bochner--Riesz problem \cite{GOWZ,Wu}, the Schr\"odinger maximal problem \cite{DGL}, and cone restriction \cite{OW}.

We adapt the methods of \cite{HR,HZ} to the cone framework of \cite{OW}. In Section~\ref{iteration}, we recast the inductive argument as a recursive procedure and combine it with the nested polynomial Wolff axioms. The main geometric distinction from the paraboloid argument is the scale to which a terminal wave-packet family is traced back. At scale $r$, the cone has approximately $r^{(n-2)/2}$ angular sectors, whereas the paraboloid has approximately $r^{(n-1)/2}$ caps. To match the cone sector count, we trace each terminal family back to its $(n-1)$-dimensional ancestor rather than to the root.

\subsection{Structure of the article} We use several constructions and estimates from Guth \cite{G2}, Ou--Wang \cite{OW}, Hickman--Rogers \cite{HR}, and Hickman--Zahl \cite{HZ}. Section~\ref{prelim} recalls the $k$-broad norm, wave packet decomposition, polynomial partitioning, transverse equidistribution, and the nested polynomial Wolff axioms. Section~\ref{iteration} develops the recursive polynomial structural decomposition. Finally, Section~\ref{conclude the proof} carries out the exponent calculations for Proposition~\ref{t2.22}.

\subsection{Notational Conventions}\label{Notation}
In this paper, we work with Schwartz functions $f$ and $g$ supported in 
$2\overline{B}^{\,n-1} \setminus B^{n-1}$. We write $x \in \mathbb{R}^n$ as 
$x=(x',x_n)$. The parameters $n$, $k$, and $\epsilon$ are fixed, and all implicit 
constants are allowed to depend on them.

We inductively define $\delta$ and $\delta_j$ for $0 \le j \le n$ by
\[
\delta_0 := \epsilon^{10}, 
\qquad 
\delta_j := \delta_{j-1}^{10} \quad (1 \le j \le n), 
\qquad 
\delta := \delta_n^{10}.
\]

We write $A \lesssim_L B$, $B \gtrsim_L A$, or $A = O_L(B)$ if 
$A \le C_L B$ for some constant $C_L$ depending only on $L$. 
We write $A \sim_L B$ if both $A \lesssim_L B$ and $A \gtrsim_L B$ hold. 
We write $A \ll B$ or $B \gg A$ if $C A \le B$ for some admissible constant 
$C \ge 1$ that may be taken arbitrarily large. 
We write $A \lessapprox B$ if $A \lesssim R^{O(\delta_0)} B$.

Let $\#A$ denote the cardinality of a finite set $A$. 
Let $N_r E$ denote the $r$-neighborhood of a set $E$. 
Let $B_r \subset \mathbb{R}^n$ denote an arbitrary ball of radius $r$. 
Let $T_xZ$ denote the tangent space of $Z$ at $x$. For a measurable set $U$ of positive finite measure, write
\[
\|h\|_{L^2_{\mathrm{avg}}(U)}^2:=\frac{1}{|U|}\int_U |h|^2.
\]
Empty products are interpreted as $1$, and empty sums as $0$.

We say that $A' \subseteq A$ is a refinement of $A$ if 
$\#A \lessapprox \#A'$. 
We write $A=\mathrm{RapDec}(r)$ if $|A|\lesssim_N r^{-N}$ for every $N>0$. 
We say that a function $F$ is essentially supported in $\Omega$ if
\[
|F(x)| \lesssim \mathrm{RapDec}(R)\,\|f\|_2
\quad \text{for all } x \notin \Omega.
\]

\subsection{Acknowledgements} The author would like to thank Xiaochun Li for suggesting the problem and for numerous insightful discussions. The author is also grateful to Larry Guth and Shunhua Shizuo for their valuable guidance in learning decoupling theory.

\subsection{AI Usage} During the preparation of the revised version of this manuscript, the author used ChatGPT to assist with language editing and to check the consistency of some mathematical arguments and computations. The mathematical results, ideas, and original proofs were developed by the author independently, prior to the use of any AI tools. All AI-generated suggestions were reviewed and verified by the author, who takes full responsibility for the content of the manuscript.

\subsection{Funding} This research received no external funding.

\bigskip

\section{Preliminaries}\label{prelim}
\subsection{Wave Packet Decomposition}\label{Wave Packet} Wave packet decomposition is a central tool in our argument. We briefly recall its definition and basic properties from Sections 2.2 and 5.2.1 of \cite{OW}.

\textbf{Definition.} Fix $R \gg 1$ and let $R^{\delta_0} \le r \le R$. Cover $2\overline{B}^{\,n-1} \setminus B^{n-1}$ by finitely overlapping $r^{-1/2}$-sectors $\theta \in \Theta[r]$. A sector $\theta \subset 2\overline{B}^{\,n-1} \setminus B^{n-1}$ is called an $r^{-1/2}$-sector if it has unit length in the radial direction and angular width $r^{-1/2}$. We denote by $\Theta[r]$ the collection of all such finitely overlapping $r^{-1/2}$-sectors. Given a Schwartz function $f$ with $\operatorname{supp} f \subset 
2\overline{B}^{\,n-1} \setminus B^{n-1}$, we write
\[
f=\sum_{\theta \in \Theta[r]} f_\theta, 
\qquad f_\theta := f\,\psi_\theta,
\]
where $\{\psi_\theta\}$ is a smooth partition of unity subordinate to the cover 
$\{\theta\}$.
For each $\theta$, we further decompose $f_\theta$ as follows. First, decompose 
$\mathbb{R}^{n-1}$ into finitely overlapping cubes of side length $r^{1/2+\delta}$ 
centered at $v \in r^{1/2+\delta}\mathbb{Z}^{n-1}$. Then
\[
\widehat{f_\theta} = \sum_v \widehat{f_\theta}\,\eta_v,
\]
where $\{\eta_v\}$ is a smooth partition of unity subordinate to this cover.

Next, for each $v$, we decompose $\widehat{f_\theta}\eta_v$ into finer pieces 
adapted to $\theta$. Cover the cube centered at $v$ by parallel thin plates 
$P_{\theta,v}^l$ of radius $r^{1/2+\delta}$, thickness $r^{\delta/2}$, and normal 
direction $\xi_\theta$, where $\xi_\theta$ denotes the unit vector pointing in the 
central direction of the sector $\theta$. Hence,
\[
f=\sum_{(\theta,v,l)\in\mathbb{T}[r]} 
\big(\widehat{f_\theta}\,\eta_v\,\eta_{\theta,v}^l\big)^{\vee},
\]
where $\{\eta_{\theta,v}^l\}$ is a smooth partition of unity subordinate to this 
cover, and $\mathbb{T}[r]$ denotes the collection of wave packets at scale $r$.

Since $\big(\widehat{f_\theta}\,\eta_v\,\eta_{\theta,v}^l\big)^{\vee}$ is 
essentially supported in $\theta$ (more precisely in $C\theta$ for some constant 
$C$), we choose bump functions $\psi_\theta'$ such that $\psi_\theta' = 1$ on an 
$r^{-1/2}$-neighborhood of $\operatorname{supp}\psi_\theta$ and decays rapidly 
outside this region. We then define
\[
f_{\theta,v}^l := 
\psi_\theta'\big(\widehat{f_\theta}\,\eta_v\,\eta_{\theta,v}^l\big)^{\vee}.
\]
Consequently,
\begin{equation}\label{wave-packet-decompose}
f = \sum_{(\theta,v,l)\in\mathbb{T}[r]} f_{\theta,v}^l 
\;+\; \mathrm{RapDec}(R)\,\|f\|_2.
\end{equation}
For any $\mathcal{T} \subseteq \mathbb{T}[r]$, define
\[
f_{\mathcal{T}} := \sum_{(\theta,v,l)\in\mathcal{T}} f_{\theta,v}^l.
\]
We say that $f$ is concentrated on the wave packets from 
$\mathcal{T} \subseteq \mathbb{T}[r]$ if
\[
\|f_{\mathbb{T}[r]\setminus\mathcal{T}}\|_{\infty}
\lesssim \mathrm{RapDec}(R)\,\|f\|_2.
\]

\textbf{$L^2$-Orthogonality. }By Plancherel's theorem, the functions $f_{\theta,v}^l$ are almost orthogonal in $L^2$.
\begin{lemma} [\cite{OW}, Section 2.2] \label{l2-orthogonality} Suppose $r \geq 1$. For any $\mathcal{T} \subseteq \mathbb{T}[r]$, we have
\[\|f_\mathcal{T}\|_2^2 \sim \sum_{(\theta, v, l) \in \mathcal{T}} \|f_{\theta, v}^l\|_2^2\]
\end{lemma} 
As in the paraboloid case \cite{HR}, there are corresponding local versions of 
the $L^2$-orthogonality for the cone, which are obtained by localizing the almost-orthogonality argument:
To justify these local forms, insert smooth cutoffs supported in $C\theta^*$ and equal to one on $\theta^*$. Since the Fourier supports of distinct wave packets have bounded overlap after this enlargement, Plancherel's theorem gives the two inequalities, with a rapidly decaying error that is absorbed into our $\mathrm{RapDec}(R)$ convention.
\begin{corollary} \label{local-l2-orthogonality-1}  Let $1 \leq \rho \leq r$. For any set $\mathcal{T} \subseteq \mathbb{T}[r]$ and $\rho^{-1/2}$-sector $\theta^*$, we have \[\|\sum_{\left(\theta ,v,l\right)\in \mathcal{T}}{f^l_{\theta ,v}}\|^2_{L^2\left({\theta }^*\right)}\lesssim \sum_{\left(\theta ,v,l\right)\in \mathcal{T}}{{\left\|f^l_{\theta ,v}\right\|}^2_{L^2\left(C{\theta }^*\right)}}\]
\[\sum_{\left(\theta ,v,l\right)\in \mathcal{T}}{{\left\|f^l_{\theta ,v}\right\|}^2_{L^2\left({\theta }^*\right)}}\lesssim {\left\|\sum_{\left(\theta ,v,l\right)\in \mathcal{T}}{f^l_{\theta ,v}}\right\|}^2_{L^2\left(C{\theta }^*\right)}\]
for some constant $C$.
\end{corollary}
\begin{corollary} \label{local-l2-orthogonality-2}  Let $1 \leq \rho \leq r$. For any set $\mathcal{T} \subseteq \mathbb{T}[r]$, we have 
\[{\mathop{\mathrm{max}}_{{\theta }^*\mathrm{:}{\rho }^{-1/2}\mathrm{-}\mathrm{sector}} {\left\|\sum_{\left(\theta ,v,l\right)\in \mathcal{T}}{f^l_{\theta ,v}}\right\|}^2_{L^2\left({\theta }^*\right)}\ }\lesssim {\mathop{\mathrm{max}}_{{\theta }^*\mathrm{:}{\rho }^{-1/2}\mathrm{-}\mathrm{sector}} \sum_{\left(\theta ,v,l\right)\in \mathcal{T}}{{\left\|f^l_{\theta ,v}\right\|}^2_{L^2\left(C{\theta }^*\right)}}\ }\]
\[{\mathop{\mathrm{max}}_{{\theta }^*\mathrm{:}{\rho }^{-1/2}\mathrm{-}\mathrm{sector}} \sum_{\left(\theta ,v,l\right)\in \mathcal{T}}{{\left\|f^l_{\theta ,v}\right\|}^2_{L^2\left({\theta }^*\right)}}\ }\lesssim {\mathop{\mathrm{max}}_{{\theta }^*\mathrm{:}{\rho }^{-1/2}\mathrm{-}\mathrm{sector}} {\left\|\sum_{\left(\theta ,v,l\right)\in \mathcal{T}}{f^l_{\theta ,v}}\right\|}^2_{L^2\left(C{\theta }^*\right)}\ }\]
for some constant $C$.
\end{corollary} 
\textbf{Spatial Concentration.} For a wave packet $(\theta,v,l)$, the restriction of $E f_{\theta,v}^l$ to the ball 
$B(0,r)$ is essentially supported in a tube. More precisely, define
\[
T_{\theta,v}^l 
:= N_{r^{\delta}}\!\left(P_{\theta,v}^l + rL(\theta)\right),
\qquad
rL(\theta) := \{\, t\,L(\theta) : -r \le t \le r \,\},
\]
where each plate is embedded in $\mathbb{R}^{n-1}\times\{0\}$. This is an $r$-tube with long direction $L(\theta)$ and short direction 
$M(\theta)$, where
\[
L(\theta) := (-\xi_\theta,\,1),
\qquad
M(\theta) := (\xi_\theta,\,1).
\]
It is well known that $E f_{\theta,v}^l$ is essentially supported in 
$T_{\theta,v}^l$. 

\subsection{A standard $L^2$ estimate} We also use the following standard $L^2$ estimate; 
see \cite{G2, OW} for details.
\begin{lemma} \label{l2.4} Let $f$ be a Schwartz function supported in $2\overline{B}^{n-1}\setminus B^{n-1}$. Then
\[{\left\|Ef\right\|}_{L^2\left(B_R\right)}\lesssim R^{1/2}{\left\|f\right\|}_2\]
Moreover, if $R^{1/2+\delta}\leq r\leq R$ and $f$ is concentrated on $R$-scale wave packets $(\theta,v,l)$ whose tubes intersect $B_r$, then
\[{\left\|Ef\right\|}_{L^2\left(B_{10r}\right)}\sim r^{1/2}{\left\|f\right\|}_2\] 
\end{lemma}

\subsection{Comparing Wave Packet Decompositions at Different Scales} Let $r^{1/2+2\delta} \le \rho \le r$ and let $y \in B(0,r)$. Suppose $g$ is 
concentrated on an $r$-scale wave packet $(\theta,v,l)$ with 
$T_{\theta,v}^l \cap B(y,\rho) \neq \emptyset$. To perform induction on scales, we decompose $g$ into wave packets at the smaller scale $\rho$ inside the ball $B(y,\rho)$. We begin by recentering $B(y,\rho)$ at the origin via the modulation
\[\tilde{g}(\xi) := e^{i\psi_y(\xi)}\, g(\xi),\]
where
\[{\psi }_y\left(\xi \right):= y_1{\xi }_1+\dots +y_{n-1}{\xi }_{n-1}+y_n|\xi|\]
One verifies that $Eg(x)=E\tilde{g}(\tilde{x})$ whenever $x=y+\tilde{x}$ with 
$\tilde{x}\in B(0,\rho)$. We then decompose $\tilde{g}$ into wave packets at scale $\rho$.
\[\tilde{g}=\sum_{\left(\zeta ,w,L\right)\in \mathbb{T}\left[\rho \right]}{{\tilde{g}}^L_{\zeta ,w}}+\mathrm{RapDec}\left(r\right){\left\|g\right\|}_2\]
Let $T^L_{\zeta,w}(y)$ denote the $\rho$-tube associated with $(\zeta,w,L)$ in the original $x$-coordinates. Define
\[{\widetilde{\mathbb{T}}}_{\theta ,v,l}:= \left\{\left(\zeta ,w,L\right):\ \mathrm{Dist}\left(\theta ,\zeta \right)\lesssim {\rho }^{-1/2},\mathrm{Dist}\left(P^L_{\zeta ,w},P^l_{\theta ,v}+P^0_{\theta }-{\partial }_{\xi }{\psi }_y\left({\xi }_{\theta }\right)\right)\lesssim r^{\delta }\right\}\] 
where 
\[\mathrm{Dist}\left(A,B\right):= {\mathop{\mathrm{max}}_{a\in A} {\mathop{\mathrm{min}}_{b\in B} \mathrm{dist}\left(a,b\right)\ }\ }\ \mathrm{for}\ A,B\subseteq {\mathbb{R}}^n\]
and $P^0_{\theta}$ denotes the thin plate centered at the origin, with radius 
$\sim r^{1/2+\delta}$, thickness $\sim r^{\delta/2}$, and normal direction 
$\xi_{\theta}$.
\begin{lemma} [\cite{OW}, Lemma 5.4] \label{l2.5} For any $\left(\zeta ,w,L\right)\in {\widetilde{\mathbb{T}}}_{\theta ,v,l}$, we have 
\[\angle \left(L\left(\theta \right),L\left(\zeta \right)\right)\lesssim {\rho }^{-1/2},\angle \left(M\left(\theta \right),M\left(\zeta \right)\right)\lesssim {\rho }^{-1/2}\]
\[\mathrm{Dist}\left(T^L_{\zeta ,w}\left(y\right),\ \left[T^l_{\theta ,v}\cap B\left(y,\rho \right)\right]+P^0_{\theta }\right)\lesssim r^{\delta }\]
where $\angle(v_1,v_2)$ denotes the angle between two non-zero vectors $v_1,v_2\in {\mathbb{R}}^n$.
\end{lemma}
The relationship between these two wave packet decompositions was studied in 
detail in \cite{OW}. In particular, if $g$ is concentrated on wave packets from 
$\mathcal{T} \subseteq \mathbb{T}[r]$, one may ask which wave packets at scale 
$\rho$ the function $\tilde{g}$ is concentrated on. We recall the relevant 
results from \cite{OW} below.
\begin{lemma} [\cite{OW}, Lemma 5.3] \label{l2.6} If $g$ is concentrated on wave packets from $\mathcal{T}\subseteq \mathbb{T}\left[r\right]$, then $\tilde{g}$ is concentrated on wave packets in $$
\widetilde{\mathcal{T}}=\bigcup_{\left(\theta ,v,l\right)\in \mathcal{T}}{{\widetilde{\mathbb{T}}}_{\theta ,v,l}}$$ 
\end{lemma}
Roughly speaking, Lemma \ref{l2.6} asserts that $(g_{\theta,v}^l)^{\widetilde{}}$ is 
concentrated on scale $\rho$ wave packets that are parallel to $T_{\theta,v}^l$ 
and contained in $c\,T_{\theta,v}^l$ for some constant $c$.

\subsection{Medium Tubes} Next, we group $\mathbb{T}[r]$ and $\mathbb{T}[\rho]$ into subcollections. For a $\rho^{-1/2}$-sector $\zeta_0$ and $v_0 \in r^{1/2+2\delta}\mathbb{Z}^{n-1} \cap B(0,\rho)$, let $\mathbb{T}_{\zeta_0,v_0}(y)$ denote the collection of $r$-tubes $T_{\theta,v}^l$ such that
\[{\mathrm{Dist}}_H\left({\zeta }_0,\theta \right)\lesssim {\rho }^{-1/2},\emptyset \neq T^l_{\theta ,v}\cap B\left(y,\rho \right)\subseteq B\left(v_0,r^{1/2+2\delta }\right)+\rho L\left({\zeta }_0\right)+\left\{y\right\}\]
where $\mathrm{Dist}_H$ denotes the Hausdorff distance. If an $r$-tube $T_{\theta,v}^l$ satisfies the above conditions for multiple pairs $(\zeta_0,v_0)$, we assign $T_{\theta,v}^l$ to only one $\mathbb{T}_{\zeta_0,v_0}(y)$. We also define
\[g_{{\zeta }_0,v_0}:= \sum_{\left(\theta ,v,l\right)\in {\mathbb{T}}_{{\zeta }_0,v_0}\left(y\right)}{g^l_{\theta ,v}}\] 
\[{\widetilde{\mathbb{T}}}_{{\zeta }_0,v_0}:= \bigcup_{\left(\theta ,v,l\right)\in {\mathbb{T}}_{{\zeta }_0,v_0}\left(y\right)}{{\widetilde{\mathbb{T}}}_{\theta ,v,l}},\ \ {\tilde{g}}_{{\zeta }_0,v_0}:= \sum_{\left(\zeta ,w,L\right)\in {\widetilde{\mathbb{T}}}_{{\zeta }_0,v_0}}{{\tilde{g}}^L_{\zeta ,w}}\] 
We immediately conclude that for all $(\zeta ,w, L)\in {\widetilde{\mathbb{T}}}_{{\zeta }_0,v_0}$
\[{\mathrm{Dist}}_H({\zeta }_0,\zeta)\lesssim {\rho }^{-\frac{1}{2}},P^L_{\zeta ,w}\subseteq B(v_0,r^{\frac{1}{2}+2\delta })\]
and
\[g=\sum_{{\zeta }_0,v_0}{g_{{\zeta }_0,v_0}}+\mbox{RapDec}\left(r\right){\left\|g\right\|}_2\]
Since each $(\zeta,w,L)$ belongs to only $O(1)$ of the collections $\widetilde{\mathbb{T}}_{\zeta_0,v_0}$, Lemma \ref{l2-orthogonality} implies 
\begin{equation}\label{medium-l2-orthogonality}
{\left\|g\right\|}^2_2\sim \sum_{{\zeta }_0,v_0}{{\left\|g_{{\zeta }_0,v_0}\right\|}^2_2},\ \ {\left\|\tilde{g}\right\|}^2_2\sim \sum_{{\zeta }_0,v_0}{{\left\|{\tilde{g}}_{{\zeta }_0,v_0}\right\|}^2_2}
\end{equation}
It follows from Lemmas \ref{l2.5} and \ref{l2.6} that
\begin{corollary}  [\cite{OW}, Lemma 5.6] \label{c2.7} If $g$ is concentrated on wave packets in ${\mathbb{T}}_{{\zeta }_0,v_0}\left(y\right)$, then $\tilde{g}$ is concentrated on wave packets in ${\widetilde{\mathbb{T}}}_{{\zeta }_0,v_0}$. On the other hand, if $\tilde{g}$ is concentrated on wave packets in ${\widetilde{\mathbb{T}}}_{{\zeta }_0,v_0}$, then inside $B\left(y,\rho \right)$\, $g$ is concentrated on wave packets in
\[\bigcup_{ \begin{array}{c}
{\mathrm{Dist}}_H\left({\zeta }_0,\zeta \right)\lesssim {\rho }^{-1/2} \\ 
{\mathrm{Dist}}_H\left(v_0,v\right)\lesssim r^{1/2+2\delta } \end{array}
}{{\mathbb{T}}_{\zeta ,v}\left(y\right)}\]
\end{corollary}
\subsection{Tangential and Transverse Wave Packets} We recall several definitions and properties of tangential and transverse wave 
packets from \cite{G2, OW}.
\begin{definition} [Transverse complete intersection, \cite{G2}, Section 5.1] Suppose $1\le m\le n$. Let $Z\left(P_1,\dots ,P_{n-m}\right)$ denote the set of common zeros of the polynomials $P_1,\dots ,P_{n-m}$ on ${\mathbb{R}}^n$. The variety $Z\left(P_1,\dots ,P_{n-m}\right)$ is said to be a transverse complete intersection if
\[\nabla P_1\left(x\right)\wedge \dots \wedge \nabla P_{n-m}\left(x\right)\neq 0,\ \ \forall x\in Z\left(P_1,\dots ,P_{n-m}\right)\] 
The degree of the transverse complete intersection is defined as ${\mathop{\mathrm{max}}_{1\le j\le n-m} {\mathrm{deg} P_j\ }\ }$. 
\end{definition} 
From the above definition, if $\mathrm{Z} := Z(P_1,\dots,P_{n-m})$ is a transverse complete intersection, then $\mathrm{Z}$ is a smooth manifold of dimension $m$. In particular, $\mathrm{Z}$ admits a tangent space $T_x\mathrm{Z}$ at every point $x \in \mathrm{Z}$.
\begin{definition} [Tangential and transverse tubes, \cite{G2}, Section 5.2] Let $\boldsymbol{\mathrm{Z}}$ be a transverse complete intersection of dimension $m$. For a fixed ball $B\left(y,r\right)$, an $r$-tube $T^l_{\theta ,v}$ meeting $B(y,r)\cap N_{r^{1/2+\delta_m}}(\boldsymbol{\mathrm Z})$ is said to be tangent to $\mathrm{Z}$ in $B\left(y,r\right)$ if $T^l_{\theta ,v}\cap B\left(y,2r\right)\subseteq N_{r^{1/2+{\delta }_m}}\left(\boldsymbol{\mathrm{Z}}\right)$ and $$\angle \left(L\left(\theta \right),T_x\boldsymbol{\mathrm{Z}}\right)\lesssim r^{-1/2+{\delta }_m},\ \ \forall x\in \mathrm{Z}\cap B\left(y,r\right)\cap N_{r^{1/2+{\delta }_m}}\left(T^l_{\theta ,v})\right)$$ Otherwise, we say that $T^l_{\theta ,v}$ is transverse to $\boldsymbol{\mathrm{Z}}$ in $B\left(y,r\right)$. We use $\mathbb{T}\left[\boldsymbol{\mathrm{Z}},B\left(y,r\right)\right]$ to denote the set of $r$-tubes tangent to $\boldsymbol{\mathrm{Z}}$ in $B\left(y,r\right)$. 
\end{definition} 
In Section \ref{iteration}, we require the following result, which can be viewed as a continuum analog of Bézout's theorem, to control the transverse intersections between a tube and an algebraic variety.
\begin{lemma} [\cite{G2}, Lemma 5.7] \label{l2.10} Let $T$ be a cylinder of radius $r$ with central line $l$ and suppose that $Z=Z\left(P_1,\dots ,P_{n-m}\right)\subseteq {\mathbb{R}}^n$ is a transverse complete intersection with degree at most $D$. For any $\alpha >0$, define
\[Z_{>\alpha }:= \left\{x\in Z:\angle \left(l,T_x\boldsymbol{\mathrm{Z}}\right)>\alpha \right\}\]
Then $Z_{>\alpha }\cap T$ is contained in a union of $\lesssim D^n$ balls of radius $\lesssim r{\alpha }^{-1}$.
\end{lemma}
\subsection{Reduction to Broad Estimates}\label{board} In \cite{G2}, Guth decomposed $\|Ef\|_{L^p(B_R)}$ into a broad part and a narrow part. The narrow part is essentially supported in a lower-dimensional subspace of $\mathbb{R}^n$. Using decoupling \cite{BD} together with induction on scales, he reduced the linear restriction estimates to estimating the broad part. Ou--Wang \cite{OW} followed the same strategy.

Fix a large constant $K$. Decompose $2\overline{B}^{\,n-1} \setminus B^{n-1}$ in frequency space into $K^{-1}$-sectors $\tau$. Given a Schwartz function $f$ with $\operatorname{supp} f \subset 2\overline{B}^{\,n-1} \setminus B^{n-1}$, write
\[f=\sum_{\tau} f_{\tau}, \qquad f_{\tau}:= f\,\psi_{\tau},\]
where $\{\psi_{\tau}\}_{\tau}$ is a smooth partition of unity subordinate to the cover $\{\tau\}$. Cover $B_R$ by finitely overlapping $B_{K^2}$. Fix a parameter $1<A\lesssim K^{\epsilon }$. For each $B_{K^2}$, define $\mu _{Ef}(B_{K^2})$ by
\[\mu _{Ef}(B_{K^2}):= {\mathop{\mathrm{min}}_{V_1,\dots ,\ V_A\ \left(k-1\right)\mathrm{-}\mathrm{subspace\ of}\ {\mathbb{R}}^n} \left(\mathop{\mathrm{max}}_{\tau :\angle \left(L\left(\tau \right),V_a\right)\gtrsim K^{-2},\forall 1\le a\le A}\int_{B_{K^2}}{{\left|Ef_{\tau }\right|}^p}\right)\ }\]
where
\[\angle \left(L\left(\tau \right),V_a\right):= {\mathop{\mathrm{min}}_{v\in L\left(\tau \right)} \angle \left(v,V_a\right)\ }\] 
For $U\subseteq {\mathbb{R}}^n$ the $k$-broad norm over $U$ is defined by
\[{\left\|Ef\right\|}_{BL^p_{k,A}\left(U\right)}:= {\left(\sum_{B_{K^2}}{\frac{\left|B_{K^2}\cap U\right|}{\left|B_{K^2}\right|}}{\mu }_{Ef}\left(B_{K^2}\right)\right)}^{\frac{1}{p}}\]
We recall several basic properties of the broad norm.
\begin{lemma} [Triangle inequality, \cite{OW}, Lemma 3.1] \label{l2.17} Suppose that $1\le p<\infty $, $f=g+h$, and $A=A_1+A_2$, where $A,A_1,A_2$ are non-negative integers. Then
\[{\left\|Ef\right\|}_{BL^p_{k,A}\left(U\right)}\lesssim {\left\|Eg\right\|}_{BL^p_{k,A_1}\left(U\right)}+{\left\|Eh\right\|}_{BL^p_{k,A_2}\left(U\right)}\] 
\end{lemma}
\begin{lemma} [H\"{o}lder's inequality, \cite{OW}, Lemma 3.2] \label{l2.18} Suppose $1\le p,p_1,p_2<\infty $, and $0\le {\alpha }_1,{\alpha }_2\le 1$ satisfying ${\alpha }_1+{\alpha }_2=1$ and
\[\frac{1}{p}=\frac{{\alpha }_1}{p_1}+\frac{{\alpha }_2}{p_2}\] 
If $A=A_1+A_2$, then
\[{\left\|Ef\right\|}_{BL^p_{k,A}\left(U\right)}\le {\left\|Ef\right\|}^{{\alpha }_1}_{BL^{p_1}_{k,A_1}\left(U\right)}{\left\|Ef\right\|}^{{\alpha }_2}_{BL^{p_2}_{k,A_2}\left(U\right)}\] 
\end{lemma}
\begin{lemma} [Standard $L^2$-estimate, \cite{OW}, Section 5]\label{l2.19} For $f\in L^2\left(2{\overline{B}}^{n-1}\backslash B^{n-1}\right)$ and non-negative integers $A$, we have
\[{\left\|Ef\right\|}^2_{BL^2_{k,A}\left(B_r\right)}\lesssim r{\left\|f\right\|}^2_2\] 
\end{lemma}
\begin{lemma} [Vanishing property, \cite{OW}, Lemma 4.5]\label{l2.20} Suppose $Ef$ is tangent to a transverse complete intersection $\boldsymbol{Z}$ in $B_r$. If ${\mathrm{deg} \boldsymbol{Z}\lesssim 1\ }$ and $\dim\boldsymbol{Z}\leq k-1$, then
\[{\left\|Ef\right\|}_{BL^p_{k,A}\left(B_r\right)}\lesssim \mathrm{RapDec}\left(r\right){\left\|f\right\|}_2\] 
\end{lemma}
The broad-to-linear reduction takes the following form.
\begin{theorem} [\cite{OW}, Section 6.2]
Let $n \geq 3$ and $2 \leq q \leq p$. Suppose for some integer $A>0$ the estimate
\begin{equation}\label{boardform}
\| Ef \|_{BL_{k,A}^p(B_R)} \lesssim_{p, q, \varepsilon} R^\varepsilon \|f\|_{L^2(\mathbb{R}^{n-1})}^\frac{2}{q}\|f\|_{L^\infty(\mathbb{R}^{n-1})}^{1-\frac{2}{q}}
\end{equation}
holds uniformly for all Schwartz functions $f$ with $\operatorname{supp} f \subset 2\overline{B}^{\,n-1} \setminus B^{n-1}$ and $R \geq 1$.  If there exists an integer $k$ with $2 \leq k \leq n$ and
\[\left\{ \begin{array}{rcl}
 q' \leq \frac{n-2}{n}p & \mathrm{if}
 & k=2 \\ p \geq \frac{n}{\frac{2n-k-1}{2}-\frac{n-k+1}{q}} & \mathrm{if} & k \geq 3
 \end{array}\right.\]
then the estimate \eqref{mixed-norm-result} holds uniformly for all Schwartz functions $f$ with $\operatorname{supp} f \subset 2\overline{B}^{\,n-1} \setminus B^{n-1}$ and $R \geq 1$. 
\end{theorem}
Here and below the broad estimate is understood with the quantifiers: for every $\varepsilon>0$ there is an integer $A=A(n,k,\varepsilon)$ for which the estimate is uniform in $R$. The broad-to-linear induction in \cite{OW} first gives the strict interior of the displayed narrow range. The boundary is recovered by the bilinear interpolation argument of \cite[Section~6.2]{OW}. The case $q=2$ follows by the corresponding $L^2$ estimate and interpolation, so no endpoint assertion is obtained by formally substituting $q=2$ into an interior induction.
Hence, together with Lemma \ref{l2.18}, Proposition \ref{result} reduces to
\begin{proposition} \label{t2.22} Suppose $n \geq 4$ and $2 \le k \le n-2$. For $p=p(n,k)$, $q=q(n,k)$ and some integer $A \lesssim_{\varepsilon,n} 1$, the estimate \eqref{boardform} holds uniformly for all Schwartz functions $f$ with $\operatorname{supp} f \subset 2\overline{B}^{\,n-1} \setminus B^{n-1}$ and $R \geq 1$. 
\end{proposition}

\subsection{Transverse Equidistribution Estimates}\label{equidistribution} Using a version of the Heisenberg uncertainty principle, Guth \cite{G2} established transverse equidistribution estimates for the paraboloid. However, in the cone setting the tubes are much thinner, and $Eg$ may be concentrated on an extremely thin layer of $N_{R^{1/2+\delta_m}}(\mathbf{Z})$, so this property may fail to hold.

To overcome this geometric obstruction, Ou and Wang \cite{OW} decompose $g$ into an essential part $g_{\mathrm{ess}}$ and a tail part $g_{\mathrm{tail}}$. 
Roughly speaking, $g_{\mathrm{ess}}$ is concentrated on wave packets 
$(\theta,v,l)$ for which $M(\theta)$ is not perpendicular to the tangent 
space of $\mathbf{Z}$ at some point $z \in \mathbf{Z} \cap N_{R^{1/2+\delta_m}}(T_{\theta,v}^l)$; consequently, transverse equidistribution estimates hold for $g_{\mathrm{ess}}$. The tail term $g_{\mathrm{tail}}$ can be made negligible by choosing $A$ 
sufficiently large.

Let $\boldsymbol{\mathrm{Z}}$ be an $m$-dimensional transverse complete intersection with degree $\le d$. Let $B_{\rho }$ be a $\rho $-ball with $\rho \gg r^{1/2+{\delta }_m}$ and $B_{\rho }\cap N_{r^{1/2+{\delta }_m}}\left(\boldsymbol{\mathrm{Z}}\right)\cap B\left(0,r\right)\neq \emptyset $. Suppose $g$ is concentrated on the wave packets from $\mathbb{T}\left[\boldsymbol{\mathrm{Z}},B\left(0,r\right)\right]\cap {\mathbb{T}}_{B_{\rho }}$, where
\[{\mathbb{T}}_{B_{\rho }}:= \left\{\left(\theta ,v,l\right)\in \mathbb{T}\left[r\right],T^l_{\theta ,v}\cap B_{\rho }\neq \emptyset \right\}\] 
For a vector space $V$ spanned by ${\left\{v_i=\left(x_{i,1},x_{i,2},\dots ,-x_{i,n}\right)\right\}}^m_{i=1}$, we define $V^+$ by
\[V^+:= \left\{\sum^m_{i=1}{a_iv^+_i}:a_i\in \mathbb{R}\mathrm{,}v^+_i=\left(x_{i,1},x_{i,2},\dots ,x_{i,n}\right),\ \mathrm{for\ }1\le i\le m\right\}\] 
We cover $B_{\rho}$ by finitely overlapping balls $B$ of radius 
$r^{1/2+\delta_m}$. For a suitable constant $c>0$, if 
$cB \cap B_{\rho} \cap \mathbf{Z} \neq \emptyset$, we choose a point 
$x \in cB \cap B_{\rho} \cap \mathbf{Z}$ and denote 
$V_B := T_x\mathbf{Z}$. Let
\[{\mathcal{B}}_a:= \left\{B:\mathrm{\ }\angle \left(V^+_B,V^{\bot }_B\right)\gtrsim K^{-2}\right\}\] 
\[{\mathcal{B}}_b:= \left\{B:\mathrm{\ }\angle \left(V^+_B,V^{\bot }_B\right)\lesssim K^{-2}\right\}\] 
\[{\mathbb{T}}_{\boldsymbol{\mathrm{Z}},B}:= \left\{\left(\theta ,v,l\right)\in \mathbb{T}\left[\boldsymbol{\mathrm{Z}},B\left(0,r\right)\right]:T^l_{\theta ,v}\cap B\neq \emptyset \right\}\]
From the above definition, if $B \in \mathcal{B}_b$, then $\angle\!\big(M(\theta), V_B^{\perp}\big)$ may be small for some $(\theta,v,l) \in \mathbb{T}_{\mathbf{Z},B}$, and consequently the transverse equidistribution estimates may fail for functions concentrated on $\mathbb{T}_{\mathbf{Z},B}$. Fortunately, there are only a few sectors $\theta$ for which $(\theta,v,l) \in \mathbb{T}_{\mathbf{Z},B}$ for some $v,l$.
\begin{lemma} [\cite{OW}, Lemma 5.8] \label{l2.24} Suppose $B\in {\mathcal{B}}_b$. Let
\[\mathbb{T}\left[{\theta }'\right]:= \left\{\left(\theta ,v,l\right):\theta ={\theta }'\right\},\ \ {\mathrm{\Theta }}_B:= \left\{\theta :\mathbb{T}\left[\theta \right]\cap {\mathbb{T}}_{\boldsymbol{\mathrm{Z}},B}\neq \emptyset \right\}\] 
Then ${\mathrm{\Theta }}_B$ is contained in the union of $O\left(1\right)$ many $K^{-2}$-sectors.
\end{lemma}
Now, we recall the definition of the essential part $g_{\mathrm{ess}}$ and the tail part $g_{\mathrm{tail}}$. 
\[{\mathbb{T}}_{\mathrm{ess}}:= \left\{\left({\zeta }_0,v_0\right):\exists \left(\theta ,v,l\right)\in {\mathbb{T}}_{{\zeta }_0,v_0}\left(y\right)\mathrm{\ s.t.\ }T^l_{\theta ,v}\cap B\neq \emptyset \mathrm{\ }\mathrm{for\ some\ }B\in {\mathcal{B}}_a\right\}\] 
\[{\mathbb{T}}_{\mathrm{tail}}:= \left\{\left({\zeta }_0,v_0\right):\forall \left(\theta ,v,l\right)\in {\mathbb{T}}_{{\zeta }_0,v_0}\left(y\right)\mathrm{,\ }T^l_{\theta ,v}\cap B=\emptyset \mathrm{\ }\mathrm{for\ all\ }B\in {\mathcal{B}}_a\right\}\]
\begin{equation} \label{r2.10} g_{\mathrm{ess}}:= \sum_{\left({\zeta }_0,v_0\right)\in {\mathbb{T}}_{\mathrm{ess}}}{g_{{\zeta }_0,v_0}},\ \ g_{\mathrm{tail}}:= \sum_{\left({\zeta }_0,v_0\right)\in {\mathbb{T}}_{\mathrm{tail}}}{g_{{\zeta }_0,v_0}}
\end{equation}
Suppose $A$ is chosen to be sufficiently large. By the definition of $k$-broad norm and Lemma \ref{l2.24}, 
\[{\left\|Eg_{\mathrm{tail}}\right\|}_{BL^p_{k,A/2}\left(B_{\rho }\right)} = \mathrm{RapDec}\left(r\right){\left\|g\right\|}_2\]
Combining this with Lemma \ref{l2.17}, we have
\begin{equation} \label{r2.12}
{\left\|Eg\right\|}_{BL^p_{k,A}\left(B_{\rho }\right)}\lesssim {\left\|Eg_{\mathrm{ess}}\right\|}_{BL^p_{k,A/2}\left(B_{\rho }\right)}+\mathrm{RapDec}\left(R\right){\left\|g\right\|}_2
\end{equation}
As shown in \cite{OW}, the transverse equidistribution estimates hold for 
the essential part $g_{\mathrm{ess}}$. The precise statement of the transverse 
equidistribution estimates for the cone is as follows:
For $b\in\mathbb R^n$, write $\boldsymbol Z_b:=\boldsymbol Z+b$. If $h$ is recentered at a ball $B(y,\rho)$, we use $\widetilde h(\xi)=e^{i\psi_y(\xi)}h(\xi)$ and write $\widetilde h_{\mathcal T}$ for its projection onto the indicated $\rho$-scale wave packets.
\begin{lemma} [\cite{OW}, Lemma 5.13] \label{l2.25} Let $B_{\rho }$\ be a $\rho $-ball with $\rho \gg r^{1/2+{\delta }_m}$ and $B_{\rho }\cap N_{r^{1/2+{\delta }_m}}\left(\boldsymbol{Z}\right)\cap B\left(0,r\right)\neq \emptyset $. Suppose $g$ is concentrated on the wave packets from $\mathbb{T}\left[\boldsymbol{Z},B\left(0,r\right)\right]\cap {\mathbb{T}}_{B_{\rho }}$ and $b\in B\left(0,r^{1/2+{\delta }_m}\right)$\, then
\[{\left\|g_{\mathrm{ess,}b\mathrm{\ }}\right\|}^2_2\lesssim r^{O\left({\delta }_m\right)}{\left(r/\rho \right)}^{-\left(n-m\right)/2}{\left\|g\right\|}^2_2\] 
where $g_{\mathrm{ess,}b\mathrm{\ }}:= {\left(g_{\mathrm{ess\ }}\right)}_b$. 
\end{lemma}
To apply the nested polynomial Wolff axioms, we need to generalize Lemma \ref{l2.25}.
\begin{lemma} \label{l2.26} Let $\rho \gg r^{1/2+\delta_m}$ and suppose
\[
B(y,\rho)\cap N_{r^{1/2+\delta_m}}(\boldsymbol Z)\cap B(0,r)\neq\emptyset.
\]
Suppose $g$ is concentrated on
$\mathbb T[\boldsymbol Z,B(0,r)]\cap\mathbb T_{B(y,\rho)}$, and let
$\widetilde{\mathcal T}\subseteq\mathbb T[\boldsymbol Z_b-y,B(0,\rho)]$.
Then
\[
 \|\widetilde g_{\mathrm{ess},\widetilde{\mathcal T}}\|_2^2
 \lesssim r^{O(\delta_m)}(r/\rho)^{-(n-m)/2}
 \|g_{\uparrow\!\uparrow\widetilde{\mathcal T}}\|_2^2
 +\mathrm{RapDec}(r)\|g\|_2^2,
\]
where $\uparrow\!\uparrow\widetilde{\mathcal T}$ consists of the $r$-scale packets
$(\theta,v,l)$ for which some $(\zeta,w,L)\in\widetilde{\mathcal T}$ satisfies
\[
 \mathrm{Dist}_H(\theta,\zeta)\le C\rho^{-1/2},\qquad
 \mathrm{Dist}_H\!\left(T^L_{\zeta,w}(y),T^l_{\theta,v}\cap B(y,\rho)\right)
 \le C r^{1/2+3\delta}.
\]
The constant $C$ may be enlarged by an absolute factor without changing the estimate.
\end{lemma}
\begin{proof}
\phantomsection\label{r2.13}\label{r2.14}\label{r2.15}
Use the medium-tube decomposition from the preceding subsection and keep the
essential/tail decomposition fixed. Let $\mathfrak I$ be the set of essential
medium groups whose $\rho$-scale descendants meet $\widetilde{\mathcal T}$.
By wave-packet concentration and bounded overlap,
\[
 \|\widetilde g_{\mathrm{ess},\widetilde{\mathcal T}}\|_2^2
 \lesssim \sum_{M\in\mathfrak I}
 \|\widetilde{(g_M)}_{\mathrm{ess},\widetilde{\mathcal T}}\|_2^2
 +\mathrm{RapDec}(r)\|g\|_2^2.
\]
The medium-group transverse equidistribution estimate used in the proof of
\cite[Lemma~5.13]{OW} gives, uniformly in $M$,
\[
 \|\widetilde{(g_M)}_{\mathrm{ess},\widetilde{\mathcal T}}\|_2^2
 \lesssim r^{O(\delta_m)}(r/\rho)^{-(n-m)/2}\|g_M\|_2^2.
\]
If $M\in\mathfrak I$, Lemmas~\ref{l2.5} and~\ref{l2.6} imply that every
$r$-scale packet contributing to $g_M$ lies in the enlarged ancestor set
$\uparrow\!\uparrow\widetilde{\mathcal T}$. The width $r^{1/2+3\delta}$ absorbs
the $r^{1/2+2\delta}$ uncertainty in the definition of a medium tube and the
bounded enlargement arising from redecomposition. Hence, by the almost
orthogonality of the original medium groups,
\[
 \sum_{M\in\mathfrak I}\|g_M\|_2^2
 \lesssim \|g_{\uparrow\!\uparrow\widetilde{\mathcal T}}\|_2^2.
\]
Combining the last three displays proves the lemma. The same argument remains
valid after the finitely many ancestor operations used below, since at each
stage we enlarge the ancestor constant by a fixed factor and absorb the
resulting rapidly decaying tails into $\mathrm{RapDec}(r)\|g\|_2^2$.
\end{proof}
\subsection{Polynomial Partitioning}\label{polynomial} Polynomial partitioning plays a central role in broad estimates. In this subsection, we review the theory of polynomial partitioning. This method is the primary tool used by Guth \cite{G1,G2} and 
Ou--Wang \cite{OW} to improve restriction estimates, and it has also led to progress on several related problems, including the Bochner--Riesz conjecture \cite{Wu} and the Schrödinger maximal estimate \cite{DGL}.

For a polynomial $P:\mathbb{R}^n \to \mathbb{R}$, let $\mathrm{cell}(P)$ denote the collection of connected components of $B_r \setminus Z(P)$, which partition $B_r\setminus Z(P)$ into finitely many disjoint regions. Each $O \in \mathrm{cell}(P)$ is called a \emph{cell} cut out by $P$.

\begin{theorem} [\cite{G2}, Section 8] \label{t2.28} Fix $r\gg 1$, and $d\in {\mathbb{N}}^+$. Suppose $F\in L^1\left({\mathbb{R}}^n\right)$ is supported on $B_r\cap N_{r^{1/2+{\delta }_m}}\boldsymbol{\mathrm{Z}}$, where $\boldsymbol{\mathrm{Z}}$ is an $m$-dimensional transverse complete intersection of degree at most $d$. Then, at least one of the following cases holds:\\\\
\underbar{Cellular case}. There exists a polynomial $P:{\mathbb{R}}^n\to \mathbb{R}$ of degree $O\left(d\right)$ with the following properties:\\
(1). $\#\mathcal{O}\sim d^m$, and each $O\in \mathcal{O}$ is contained in a ball of radius $r/2$.\\
(2). For each $O\in \mathcal{O}$, ${\left\|F\right\|}_{L^1\left(O\right)}\sim d^{-m}{\left\|F\right\|}_{L^1\left(B_r\right)}$\\
where $\mathcal{O}$ is a refinement of the set of connected components of $B_r\backslash N_{r^{1/2+{\delta}_m}}Z\left(P\right)$. \\\\
\underbar{Algebraic case}. There exists a $\left(m-1\right)$-dimensional transverse complete intersection $\boldsymbol{\mathrm{Y}}\subseteq \boldsymbol{\mathrm{Z}}$ of degree $O\left(d\right)$ such that
\[{\left\|F\right\|}_{L^1\left(B_r\right)}\lesssim {\left\|F\right\|}_{L^1\left(B_r\cap N_{r^{1/2+{\delta }_m}}\boldsymbol{\mathrm{Y}}\right)}\] 
\end{theorem}
We now apply polynomial partitioning to the broad estimates. If the cellular case holds, then
\[
\|Ef\|^p_{BL^p_{k,A}\!\left(B_r \cap N_{r^{1/2+\delta_m}}\mathbf{Z}\right)}
\lesssim 
d^m \|Ef\|^p_{BL^p_{k,A}(O)}
\quad \text{for all } O \in \mathcal{O},
\]
where $\mathcal{O}$ denotes the collection of cells produced by Theorem \ref{t2.28}. If the algebraic case holds, then
\[
\|Ef\|^p_{BL^p_{k,A}\!\left(B_r \cap N_{r^{1/2+\delta_m}}\mathbf{Z}\right)}
\lesssim 
\|Ef\|^p_{BL^p_{k,A}\!\left(B_r \cap N_{r^{1/2+\delta_m}}\mathbf{Y}\right)},
\]
where $\mathbf{Y}$ is an $(m-1)$-dimensional transverse complete intersection produced by Theorem~\ref{t2.28}.

The fundamental theorem of algebra yields the following tube--cell incidence bound.
\begin{lemma} \label{l2.29} For any $r$-tube $T^l_{\theta ,v}$, 
\[\#\left\{O\in {\mathcal{O}}:O\cap T^l_{\theta ,v}\neq \emptyset \right\}\lesssim d\] 
\end{lemma}

\subsection{Nested Polynomial Wolff Axiom}\label{Kakeya} To refine the work of Ou--Wang \cite{OW}, we require the nested polynomial Wolff axioms. We begin by recalling the definitions of grains and multigrains from Section~3 of \cite{HZ}.
\begin{definition} [\cite{HZ}, Definition 3.3] A grain is defined to be a pair $\left(Z,B\left(y,r\right)\right)$ where $Z$ is a transverse complete intersection and $B_r$ is a ball of radius $r>0$. The dimension of a grain $\left(Z,B\left(y,r\right)\right)$ is that of the variety $Z$, while its degree is that of $Z$ and its scale is $r$. 
\end{definition}
\begin{definition} [\cite{HZ}, Definition 3.4] Let $S=\left(Z,B\left(y,r\right)\right)$ be a grain of dimension $m$. A function $f$ with $\mathrm{supp}\ f\subseteq 2{\overline{B}}^{n-1}\backslash B^{n-1}$ is said to be tangent to $S$ if it is concentrated on scale $r$ wave packets belonging to the collection:
\[\mathbb{T}\left[S\right]:= \left\{\left(\theta ,v,l\right)\in \mathbb{T}\left[r\right]:T^l_{\theta ,v}\left(y\right)\ \mathrm{is}\ \mathrm{tangent\ to}\ Z\ \mathrm{in}\ B\left(y,r\right)\right\}\] 
\end{definition}
\begin{definition} [\cite{HZ}, Definition 3.5] A multigrain is a tuple of grains
\[\overrightarrow{S}=\left(S_s,S_{s+1},\dots ,S_t\right),\ \ S_i=\left(Z_i,B\left(y_i,r_i\right)\right)\] 
satisfying $${\mathrm{dim} Z_i\ }=i, \quad\forall i$$$$B(y_i,r_i)\cap B(y_{i+1},r_{i+1}) \neq \emptyset, \quad \forall i $$$$ N_{Cr^{1/2+{\delta }_s}_s}Z_s\subseteq \dots \subseteq N_{Cr^{1/2+{\delta }_t}_t}Z_t \mathrm{\ for\ some\ } C>0 $$$$ 0 < r_s\le r_{s+1}\le \dots \le r_t$$
\end{definition}
The parameter $(t-s+1)$ is called the multigrain level. The degree of $\overrightarrow{S}$ is $\max_{s\leq i\leq t}\deg Z_i$. The multiscale of $\overrightarrow{S}$ is the tuple $\overrightarrow{r} = (r_s, r_{s+1}, \dots, r_t)$. Finally, we write
\[{\overrightarrow{S}}_1=\left(S_s,S_{s+1},\dots ,S_t\right)\preccurlyeq {\overrightarrow{S}}_2=\left(S_{s'},S_{s'+1},\dots ,S_t\right),\ \ \mathrm{if\ }s'\ge s\]

The restriction conjecture is closely related to a Kakeya-type problem 
concerning the continuum incidence theory of tubes. More precisely, let 
$\overrightarrow{S}_{s,t} = (S_s,\dots,S_t)$ be a multigrain with 
$S_i = (Z_i, B(y_i,r_i))$. Suppose $\mathcal{L}$ is a collection of 
$r_t^{-1/2}$-direction-separated lines $L$ satisfying
\[
L \cap B(y_i,r_i) \cap N_{\rho_i} Z_i \neq \emptyset
\]
and
\[
L \cap B(y_i,2r_i) 
\subseteq N_{\rho_i} Z_i \cap B(y_i,2r_i)
\]
for each $i$, where $\rho_i>0$ is a prescribed width. What is the largest possible cardinality of $\mathcal{L}$? 
Understanding this incidence problem is crucial for obtaining effective estimates for $Ef$. It is well known that the restriction conjecture implies the Kakeya conjecture, the latter providing a heuristic bound on the number of incidences between direction-separated tubes. For Wolff's work on Kakeya-type maximal functions, see \cite{Wol2}.

In \cite{GOWZ,GOW,HZ}, a multiscale structure, referred to as nested tubes, is introduced to describe the spatial relationships 
between tubes at different scales.
\begin{definition} [Nested Tubes] \label{d2.15} Let ${\overrightarrow{S}}_m=\left(S_m,\dots ,S_n\right)$ be a multigrain with $S_i=\left(Z_i,B\left(y_i,r_i\right)\right)$ and ${\mathrm{dim} Z_i\ }=i$. We define $\mathbb{T}\left[{\overrightarrow{S}}^i_m\right]$ by induction on $i\in\mathbb{N}\cap[m,n]$. For $i=m$, we set
\[\mathbb{T}\left[{\overrightarrow{S}}^m_m\right]:= \mathbb{T}\left[S_m\right]\ \ \mathrm{\Theta }\left[{\overrightarrow{S}}^m_m\right]:= \left\{\theta :\exists T^l_{\theta ,v}\in \mathbb{T}\left[{\overrightarrow{S}}^m_m\right]\ \mathrm{for\ some\ }v,l\right\}\] 
For $m<i\le n$, let $\mathbb{T}\left[{\overrightarrow{S}}^i_m\right]$ be the set of all $r_i$-tubes $T^l_{\theta ,v}$ satisfying the following conditions: \\
(a). $T^l_{\theta ,v}\in \mathbb{T}\left[Z_i,B\left(y_i,r_i\right)\right]\ $ \\
(b). There exists $T^L_{\zeta,w}\in\mathbb{T}[\overrightarrow{S}^{i-1}_m]$ such that $$\quad {\mathrm{Dist}}_H\left(\theta,\zeta\right)\lesssim r^{-1/2}_{i-1},\quad{\mathrm{Dist}}_H\left(T^L_{\zeta,w}(y_{i-1}),\ T^l_{\theta,v}(y_i)\cap B(y_{i-1},r_{i-1})\right)\lesssim r^{1/2+\delta }_i$$ 
and \[\ \mathrm{\Theta }\left[{\overrightarrow{S}}^i_m\right]:= \left\{\theta :\exists T^l_{\theta ,v}\in \mathbb{T}\left[{\overrightarrow{S}}^i_m\right]\ \mathrm{for\ some\ }v,l\right\}\] Finally, we set \[\mathbb{T}\left[{\overrightarrow{S}}^{\#}_m\right]:= \mathbb{T}\left[{\overrightarrow{S}}^{n-1}_m\right],\ \ \mathrm{\Theta }\left[{\overrightarrow{S}}^{\#}_m\right]:= \ \mathrm{\Theta }\left[{\overrightarrow{S}}^{n-1}_m\right]\] 
\end{definition}
The estimate below follows from \cite[Lemma~2.11]{Z}; see also \cite{HZ}. Related forms of this estimate have been used to study the Bochner--Riesz problem, Stein's square function, and paraboloid restriction; see \cite{GOWZ,GOW,HZ}.

We record why the translated multigrains produced below fall within this
framework. Enlarging all balls by a dimensional constant, consecutive
intersections allow us to replace their centers by a common center. For
$m\le i\le n-1$ put
\[
 W_i=C_n r_{n-1}^{\delta_0}r_i^{1/2},\qquad
 \mathcal S_i=B(y,C_nr_i)\cap\bigcup_{s=m}^iN_{C_nr_s^{1/2+\delta_s}}Z_s.
\]
The relations $Z_s=Z_s'+b$, $Z_s'\subseteq Z_i$, and
$|b|\lesssim r_i^{1/2+\delta_i}$ supplied by Algorithm~2 imply, after increasing
$C_n$, that $\mathcal S_m\subseteq\cdots\subseteq\mathcal S_{n-1}$ and
$\mathcal S_i\subseteq N_{C_nW_i}Z_i$. These are semialgebraic sets of bounded
complexity. The standard tubular-volume bound for algebraic varieties \cite{Won} gives
the local growth hypothesis in \cite[Lemma~2.11]{Z}. Finally, the corrected
nested-tube definition gives a compatible chain at every level; the angular
and positional errors accumulate only over at most $n$ levels and are absorbed
by the factors $C_nW_i$. Thus Zahl's lemma applies to the direction set of the
chain. This is the only form of the nested polynomial Wolff axiom used below.
\begin{lemma} [\cite{HZ}, Lemma 3.7, Nested polynomial Wolff axioms] \label{l2.161} Let ${\overrightarrow{S}}_m=\left(S_m,\dots ,S_n\right)$ be a multigrain with $S_i=\left(Z_i,B\left(y_i,r_i\right)\right)$ such that for all $i\in \mathbb{N}\cap \left[m,n\right]$, 
\[{\mathrm{dim} Z_i\ }=i\ ,\ \ {\mathrm{deg} Z_i\ }\le d \] 
we have
\[\#\mathrm{\Theta }\left[{\overrightarrow{S}}^n_m\right]{\lesssim }_dr^{O\left({\delta }_0\right)}_n\left(\prod^{n-1}_{j=m}{r^{-\frac{1}{2}}_j}\right)r^{\frac{n-1}{2}}_{n}\] 
\end{lemma}
Consequently,
\begin{corollary} \label{l2.16} Let ${\overrightarrow{S}}_m=\left(S_m,\dots ,S_n\right)$ be a multigrain with $S_i=\left(Z_i,B\left(y_i,r_i\right)\right)$ such that for all $i\in \mathbb{N}\cap \left[m,n\right]$, 
\[{\mathrm{dim} Z_i\ }=i\ ,\ \ {\mathrm{deg} Z_i\ }\le d \] 
we have
\[\#\mathrm{\Theta }\left[{\overrightarrow{S}}^{\#}_m\right]{\lesssim }_dr^{O\left({\delta }_0\right)}_{n-1}\left(\prod^{n-2}_{j=m}{r^{-\frac{1}{2}}_j}\right)r^{\frac{n-2}{2}}_{n-1}\] 
\end{corollary}
\begin{proof} Apply Lemma \ref{l2.161} to ${\overrightarrow{S}}:=(S_m,...,S_{n-1},({\mathbb{R}}^n,B(y_{n-1},r_{n-1})))$. 
\end{proof}
\bigskip

\section{Finding Polynomial Structure}\label{iteration}
In this section, we reformulate the induction arguments of \cite{OW} as a recursive procedure, combining ideas from \cite{HR, HZ}. The argument consists of two algorithms.

In Algorithm 1, we iteratively apply polynomial partitioning until the tangential term dominates. When the tangential term dominates, the procedure reduces the dimension from $m$ to $m-1$; otherwise, it terminates at a sufficiently small scale. Algorithm 2 repeatedly applies Algorithm 1 and stops when the small-scale outcome dominates. 

Throughout the iteration, we suppress errors of size $\mathrm{RapDec}(R)\|f\|_2$. Choosing the decay exponent sufficiently large makes these errors negligible.
\subsection{Algorithm 1}\label{induction-on-scale} This algorithm proceeds via iterated polynomial partitioning. At each step, one 
of three cases occurs: the cellular case, the transverse case, or the tangential case. The iteration terminates either when the tangential case occurs or when the scale becomes sufficiently small. \\\\
\underline{Input.} Algorithm~1 takes the following data as input:
\begin{itemize}
\item A grain $(\mathbf{Z}, B(y,r))$ of dimension $m$.
\item A function $f \in L^\infty\!\left(2\overline{B}^{\,n-1}\setminus B^{n-1}\right)$ 
that is tangent to $(\mathbf{Z}, B(y,r))$.
\item An admissible integer $A \in \mathbb{N}$.
\end{itemize}

At the initial step, we set
\[
O_0 := B_r\cap N_{r^{1/2+\delta_m}}\mathbf{Z},\qquad \mathcal{O}_0:=\{O_0\}, 
\qquad
f_{O_0} := f,
\qquad
\mathfrak{h}_0 := \emptyset,
\qquad
\rho_0 := r.
\]
\underline{Output.} At the $j$-th iteration, Algorithm~1 outputs:

\begin{itemize}
\item \textbf{History.} A word $\mathfrak{h}_j$ of length $j$ in the alphabet 
$\{\mathrm{a},\mathrm{c}\}$, recording how each cell $O_j \in \mathcal{O}_j$ 
is produced by the iterated polynomial partitioning. 
Let $\#\mathrm{a}(j)$ and $\#\mathrm{c}(j)$ denote the number of letters 
$\mathrm{a}$ and $\mathrm{c}$ in $\mathfrak{h}_j$, respectively.

\item \textbf{Spatial scale $\rho_j$.} The scale $\rho_j$ is determined by the 
initial scale $r$ and the history $\mathfrak{h}_j$. Define an auxiliary parameter 
$\widetilde{\delta}_{m-1}$ by
\[
(1-\widetilde{\delta}_{m-1})\!\left(\frac{1}{2}+\delta_{m-1}\right)
= \frac{1}{2}+\delta_m,
\]
which implies
\[
\frac{\delta_{m-1}}{2} \le \widetilde{\delta}_{m-1} \le 2\delta_{m-1}.
\]

\item \textbf{Cells.} A family of subsets $\mathcal{O}_j \subset \mathbb{R}^n$. 
Each cell $O_j \in \mathcal{O}_j$ is contained in a ball 
$B_{O_j}=B(y_{O_j},\rho_j)$.

\item \textbf{Localized functions.} A collection of functions 
$\{f_{O_j}\}_{O_j\in\mathcal{O}_j}$. For each cell $O_j$ there exists a 
translate $\mathbf{Z}_{O_j}:=\mathbf{Z}+x_{O_j}$ such that $f_{O_j}$ is tangent 
to the grain $(\mathbf{Z}_{O_j},B_{O_j})$.

\item \textbf{Auxiliary parameters.} A large integer $d$ depending only on the 
admissible parameters and $\deg \mathbf{Z}$, together with
\[
C^{\mathrm{I}}_{j,\delta}(d,r)
:= d^{\#\mathrm{c}(j)\delta}\,(\log r)^{\#\mathrm{a}(j)(1+\delta)},
\qquad
C^{\mathrm{II}}_{j,\delta}(d)
:= d^{\#\mathrm{c}(j)\delta + n\#\mathrm{a}(j)(1+\delta)},
\]
\[
C^{\mathrm{III}}_{j,\delta}(d,r)
:= d^{j\delta} r^{\overline{C}\#\mathrm{a}(j)\delta_m},
\qquad
C^{\mathrm{IV}}_{j,\delta}(d,r)
:= C^{\mathrm{III}}_{j,\delta}(d,r),
\qquad
A_j := 3^{-\#\mathrm{a}(j)}A.
\]
\end{itemize}
\underline{Properties.} At the $j$-th iteration:

\begin{enumerate}
\item The mass of $\|Ef\|^p_{BL^p_{k,A}(B_r)}$ is concentrated on the cells 
$O_j \in \mathcal{O}_j$:
\begin{equation}\label{r3.1}
\|Ef\|^p_{BL^p_{k,A}(B_r)}
\le C^{\mathrm{I}}_{j,\delta}(d,r)
\sum_{O_j\in\mathcal{O}_j}
\|Ef_{O_j}\|^p_{BL^p_{k,A_j}(O_j)} .
\end{equation}

\item The functions $f_{O_j}$ satisfy
\begin{equation}\label{r3.2}
\sum_{O_j\in\mathcal{O}_j} \|f_{O_j}\|_2^2
\le C^{\mathrm{II}}_{j,\delta}(d)\,
d^{\#\mathrm{c}(j)} \|f\|_2^2 .
\end{equation}

\item Each $f_{O_j}$ satisfies, for all $1\le \rho \le \rho_j$,
\begin{equation}\label{r3.3}
\|f_{O_j}\|_2^2
\le C^{\mathrm{III}}_{j,\delta}(d,r)
\left(\frac{r}{\rho_j}\right)^{-(n-m)/2}
d^{-\#\mathrm{c}(j)(m-1)} \|f\|_2^2 ,
\end{equation}
and for all $\rho\in[1,\rho_j]$,
\begin{multline} \label{r3.4}
    \max_{\theta:\,\rho^{-1/2}\text{-sector}}
\|f_{O_j}\|_{L^2_{\mathrm{avg}}(\theta)}^2
\le \\ C^{\mathrm{III}}_{j,\delta}(d,r)
\left(\frac{r}{\rho_j}\right)^{-(n-m)/2}
\max_{\theta:\,\rho^{-1/2}\text{-sector}}
\|f\|_{L^2_{\mathrm{avg}}(\theta)}^2
\end{multline}

\item For $\mathcal{T}\subseteq\mathbb{T}[\rho_j]$ and $O_j\in\mathcal{O}_j$,
\begin{equation}\label{r3.5}
\|f_{O_j,\mathcal{T}}\|_2^2
\le C^{\mathrm{IV}}_{j,\delta}(d,r)
\left(\frac{r}{\rho_j}\right)^{-(n-m)/2}
\|f_{\uparrow^j\mathcal{T}}\|_2^2 ,
\end{equation}
where $(f_{O_j,\mathcal{T}})^{\widetilde{}}:=({f_{O_j}})^{\widetilde{}}_{\mathcal{T}}$, 
and $\uparrow^j\mathcal{T}$ denotes the set of wave packets 
$(\theta,v,l)\in\mathbb{T}[r]$ such that
\[\mathrm{dist}_H(\theta,\zeta)\le c_j\rho_j^{-1/2}\]
\[\mathrm{dist}_H\!\left(T_{\zeta,w}^L(y_{O_j}),
\,T_{\theta,v}^l(y)\cap B_{O_j}\right)\le c_j r^{1/2+\delta}\]
for some $(\zeta,w,L)\in\mathcal{T}$, where
\[
c_j:=\sum_{i=1}^j 2^{-(i-1)/2} C
\]
for an absolute constant $C$.
\end{enumerate}
\underline{Recursion Steps.} If $\rho_j \le r^{\widetilde{\delta}_{m-1}}$, then the
algorithm terminates. In this case we set $J:=j$ and $\mathcal{O}:=\mathcal{O}_j$,
and label this outcome as \textup{[tiny]}. Otherwise, for each $O_j\in\mathcal{O}_j$
we apply polynomial partitioning with degree $d$ to
$\|Ef_{O_j}\|^p_{BL^p_{k,A_j}(O_j)}$.

Let $\mathcal{O}_{j,\mathrm{cell}}\subseteq \mathcal{O}_j$ denote the collection of
cells for which the cellular case occurs, and set
\[
\mathcal{O}_{j,\mathrm{alg}} := \mathcal{O}_j \setminus \mathcal{O}_{j,\mathrm{cell}}.
\]
Then one of the following cases holds:\\\\
\textbf{Cellular-dominant case.}
Suppose
\begin{equation}\label{r3.6}
\sum_{O_j\in \mathcal{O}_{j,\mathrm{alg}}}
\|Ef_{O_j}\|^p_{BL^p_{k,A_j}(O_j)}
\le 
\sum_{O_j\in \mathcal{O}_{j,\mathrm{cell}}}
\|Ef_{O_j}\|^p_{BL^p_{k,A_j}(O_j)} .
\end{equation}
Append the letter $\mathrm{c}$ to $\mathfrak{h}_j$ to obtain 
$\mathfrak{h}_{j+1}$. Then
\[
\#\mathrm{a}(j+1)=\#\mathrm{a}(j), \qquad 
\#\mathrm{c}(j+1)=\#\mathrm{c}(j)+1, \qquad 
A_{j+1}=A_j,
\]
and
\[
C^{\mathrm{I}}_{j+1,\delta}=d^{\delta}C^{\mathrm{I}}_{j,\delta},\qquad
C^{\mathrm{II}}_{j+1,\delta}=d^{\delta}C^{\mathrm{II}}_{j,\delta},
\]
\[
C^{\mathrm{III}}_{j+1,\delta}=d^{\delta}C^{\mathrm{III}}_{j,\delta},\qquad
C^{\mathrm{IV}}_{j+1,\delta}=d^{\delta}C^{\mathrm{IV}}_{j,\delta}.
\]
Set $\rho_{j+1}:=\rho_j/2$. For each $O_j\in\mathcal{O}_{j,\mathrm{cell}}$, the cellular case yields a 
polynomial $P$ of degree $O(d)$ such that:

\begin{enumerate}
\item $\#\mathcal{O}(O_j)\sim d^m$,
\item every $O\in\mathcal{O}(O_j)$ lies in a ball $B_O=B(y_O,\rho_{j+1})$,
\item $\|Ef_{O_j}\|^p_{BL^p_{k,A_j}(O)}
\sim d^{-m}\|Ef_{O_j}\|^p_{BL^p_{k,A_j}(O_j)}$.
\end{enumerate}
Here $\mathcal{O}(O_j)$ refines the connected components of 
$O_j\setminus N_{\rho_j^{1/2+\delta_m}}Z(P)$. Define
\[
f_O:=\sum_{\substack{(\theta,v,l)\in\mathbb{T}[\rho_j]\\
T_{\theta,v}^l\cap O\neq\emptyset}}
(f_{O_j})_{\theta,v}^l .
\]
By Lemma \ref{l2.29} and Lemma \ref{l2-orthogonality},
\[
\sum_{O\in\mathcal{O}(O_j)}\|f_O\|_2^2
\lesssim d\,\|f_{O_j}\|_2^2 .
\]
Pigeonholing gives a refinement $\mathcal{O}'(O_j)$ with 
$\#\mathcal{O}'\sim d^m$ such that
\[
\|f_O\|_2^2\lesssim d^{-m+1}\|f_{O_j}\|_2^2 .
\]
Since $\rho_{j+1}\sim\rho_j$, there exists a finite set 
$\mathcal{B}\subset B(0,\rho_j^{1/2+\delta_m})$ with $\#\mathcal{B}\sim1$ such that
\[
N_{\rho_j^{1/2+\delta_m}}\mathbf{Z}_{O_j}
\subset 
\bigcup_{b\in\mathcal{B}}
N_{\rho_{j+1}^{1/2+\delta_m}}(\mathbf{Z}_{O_j}+b).
\]
For $O\in\mathcal{O}'(O_j)$ and $b\in\mathcal{B}$ define
\[
O_b:=O\cap N_{\rho_{j+1}^{1/2+\delta_m}}(\mathbf{Z}_{O_j}+b),
\]
\[
\widetilde{f_{O_b}}
:=\sum_{\substack{(\zeta,w,L)\in\mathbb{T}[\rho_{j+1}]\\
T_{\zeta,w}^L(y_O)\cap O_b\neq\emptyset}}
(\widetilde{f_O})_{\zeta,w}^L ,
\qquad
\widetilde{f_O}(\xi)=e^{i\psi_{y_O}(\xi)}f_O(\xi).
\]
By Lemma \ref{l2.6}, $f_{O_b}$ is tangent to 
$(\mathbf{Z}_{O_j}+b,B_O)$. Set
\[
\mathcal{O}_{j+1}
:=\bigcup_{O_j\in\mathcal{O}_{j,\mathrm{cell}}}
\bigcup_{O\in\mathcal{O}'(O_j)}
\{O_b:\,b\in\mathcal{B}\}.
\]
Properties 1--3 follow from \cite{HR}. We verify Property~4.  
By Lemma \ref{l2-orthogonality} and Lemma \ref{l2.6}, for 
$\mathcal{T}\subseteq \mathbb{T}[\rho_{j+1}]$,
\begin{equation}\label{r3.7}
\|\widetilde{f}_{O_{j+1},\mathcal{T}}\|_2^2
\lesssim 
\|f_{O_j,\uparrow \mathcal{T}}\|_2^2 .
\end{equation}
Here $\uparrow\mathcal{T}$ denotes the set of wave packets 
$(\theta,v,l)\in\mathbb{T}[\rho_j]$ such that
\[
\mathrm{Dist}_H(\theta,\zeta)\le C\rho_{j+1}^{-1/2},
\qquad
\mathrm{Dist}_H\!\left(
T_{\zeta,w}^L(y_{O_{j+1}}),
\,T_{\theta,v}^l\cap B(y_{O_{j+1}},\rho_{j+1})
\right)
\le C\rho_j^{1/2+\delta}
\]
for some $(\zeta,w,L)\in\mathcal{T}$, where $C$ is the absolute constant from Property 4. Combining \eqref{r3.5} and \eqref{r3.7}, we obtain
\[
\|f_{O_{j+1},\mathcal{T}}\|_2^2
\lesssim 
d^{-\delta} C^{\mathrm{IV}}_{j+1,\delta}(d,r)
\left(\frac{r}{\rho_j}\right)^{-(n-m)/2}
\|f_{\uparrow^j(\uparrow\mathcal{T})}\|_2^2 .
\]
Since $\rho_{j+1}\sim\rho_j$, choosing $d$ sufficiently large yields
\[
\|f_{O_{j+1},\mathcal{T}}\|_2^2
\le 
C^{\mathrm{IV}}_{j+1,\delta}(d,r)
\left(\frac{r}{\rho_{j+1}}\right)^{-(n-m)/2}
\|f_{\uparrow^j(\uparrow\mathcal{T})}\|_2^2 .
\]
The ancestor-set inclusion used in \cite{HZ} gives ${\uparrow }^j\left(\uparrow \mathcal{T}\right)\subseteq {\uparrow }^{j+1}\mathcal{T}$. This completes the verification.\\\\
\textbf{Algebraic-dominant Case. }If \eqref{r3.6} fails, then
\[\sum_{O_j\in {\mathcal{O}}_{j\mathrm{,cell}}}{{\left\|Ef_{O_j}\right\|}^p_{BL^p_{k,A_j}\left(O_j\right)}}\le \sum_{O_j\in {\mathcal{O}}_{j\mathrm{,alg}}}{{\left\|Ef_{O_j}\right\|}^p_{BL^p_{k,A_j}\left(O_j\right)}}\]
Append the letter $\mathrm{a}$ to $\mathfrak{h}_j$ to obtain $\mathfrak{h}_{j+1}$. Now we have
\[\#\mathrm{a}\left(j+1\right)=\#\mathrm{a}\left(j\right)+1,\ \ \#\mathrm{c}\left(j+1\right)=\#\mathrm{c}\left(j\right),\ \ A_{j+1}=A_j/3\] 
\[C^{\mathrm{I}}_{j+1,\delta }\left(d,r\right)={{\mathrm{log}}^{1+\delta } r\ }C^{\mathrm{I}}_{j,\delta }\left(d,r\right),\ \ C^{\mathrm{II}}_{j+1,\delta }\left(d\right)=d^{n\left(1+\delta \right)}C^{\mathrm{II}}_{j,\delta }\left(d\right)\] 
\[C^{\mathrm{III}}_{j+1,\delta }\left(d,r\right)=r^{\overline{C}{\delta }_m}d^{\delta }C^{\mathrm{III}}_{j,\delta }\left(d,r\right),\ \ C^{\mathrm{IV}}_{j+1,\delta }\left(d,r\right)=r^{\overline{C}{\delta }_m}d^{\delta }C^{\mathrm{IV}}_{j,\delta }\left(d,r\right)\]
Set $\rho_{j+1}:=\rho_j^{\,1-\widetilde{\delta}_{m-1}}$. 
Since the algebraic case occurs, for each $O_j\in\mathcal{O}_{j,\mathrm{alg}}$ there exists an $(m-1)$-dimensional transverse complete intersection $\mathbf{Y}\subseteq \mathbf{Z}_{O_j}$ of degree $C_{\mathrm{alg}}d$ such that
\[{\left\|Ef_{O_j}\right\|}^p_{BL^p_{k,A_j}\left(O_j\right)}\lesssim {\left\|Ef_{O_j}\right\|}^p_{BL^p_{k,A_j}\left(O_j\cap N_{{\rho }^{1/2+{\delta }_m}_j}\boldsymbol{\mathrm{Y}}\right)}\] 
where $C_{\mathrm{alg}}$ is a sufficiently large constant. Cover $B_{O_j}$ by finitely overlapping balls of radius $\rho_{j+1}$, and let $\mathcal{B}_{O_j}$ denote this collection. 
For each $B\in\mathcal{B}_{O_j}$, define
\[
\mathbb{T}_B
:=\left\{(\theta,v,l)\in\mathbb{T}[\rho_j]:
T_{\theta,v}^l \cap B \cap N_{\rho_j^{1/2+\delta_m}}\mathbf{Y}\neq\emptyset
\right\},
\qquad
f_B := (f_{O_j})_{\mathbb{T}_B}.
\]
Let $\mathbb{T}_{B,\mathrm{tang}}\subseteq \mathbb{T}_B$ be the set of 
$\rho_j$-scale wave packets $(\theta,v,l)$ satisfying:

\begin{enumerate}
\item $T_{\theta,v}^l \cap 2B \subseteq 
N_{\rho_j^{1/2+\delta_m}}\mathbf{Y}
= N_{\rho_{j+1}^{1/2+\delta_{m-1}}}\mathbf{Y}$;

\item whenever $x\in T_{\theta,v}^l$ and 
$y\in \mathbf{Y}\cap 2B$ satisfy 
$|x-y|\lesssim \rho_j^{1/2+\delta_m}
=\rho_{j+1}^{1/2+\delta_{m-1}}$, then
\[
\angle\!\big(L(\theta),T_y\mathbf{Y}\big)
\lesssim \rho_{j+1}^{-1/2+\delta_{m-1}}.
\]
\end{enumerate}
Define $\mathbb{T}_{B,\mathrm{trans}}
:= \mathbb{T}_B \setminus \mathbb{T}_{B,\mathrm{tang}}$ and
\[
f_{B,\mathrm{tang}} := (f_{O_j})_{\mathbb{T}_{B,\mathrm{tang}}},
\qquad
f_{B,\mathrm{trans}} := (f_{O_j})_{\mathbb{T}_{B,\mathrm{trans}}}.
\]
By Lemma~\ref{l2.17},
\begin{equation} \label{r3.8}
{\left\|Ef_{O_j}\right\|}^p_{BL^p_{k,A_j\ }\left(O_j\right)}\lesssim \sum_{B\in {\mathcal{B}}_{O_j}}{{\left\|Ef_{B,\mathrm{tang}}\right\|}^p_{BL^p_{k,A_{j+1}}\left(B\right)}}+\sum_{B\in {\mathcal{B}}_{O_j}}{{\left\|Ef_{B,\mathrm{trans}}\right\|}^p_{BL^p_{k,2A_{j+1}}\left(B\right)}}
\end{equation}
By a pigeonholing argument, at least one of the following cases holds.\\\\
\underline{Tangential Subcase.} Let $C_{\mathrm{tang}}$ be a large constant. 
If
\begin{equation}\label{r3.9}
\sum_{O_j\in\mathcal{O}_{j,\mathrm{alg}}}
\|Ef_{O_j}\|^p_{BL^p_{k,A_j}(O_j)}
\le 
\frac{C_{\mathrm{tang}}}{2}
\sum_{O_j\in\mathcal{O}_{j,\mathrm{alg}}}
\sum_{B\in\mathcal{B}_{O_j}}
\|Ef_{B,\mathrm{tang}}\|^p_{BL^p_{k,A_{j+1}}(B)},
\end{equation}
then the algorithm terminates. We set $J=j+1$ and label this outcome \textup{[tang]}. Define the collection of grains
\[
\mathcal{S}
:=\left\{(\mathbf{Y},B):
B\in\mathcal{B}_{O_j}\ \text{for some}\ 
O_j\in\mathcal{O}_{j,\mathrm{alg}}
\right\}.
\]
For each $S=(\mathbf{Y},B)\in\mathcal{S}$, set 
$f_S:=f_{B,\mathrm{tang}}$ and $B_S:=B$. 
Lemma~\ref{l2.6} implies that $f_S$ is tangent to $S$.
Since \eqref{r3.6} fails, inequality \eqref{r3.9} yields
\begin{equation}\label{r3.10}
\sum_{O_j\in\mathcal{O}_j}
\|Ef_{O_j}\|^p_{BL^p_{k,A_j}(O_j)}
\le 
C_{\mathrm{tang}}
\sum_{S\in\mathcal{S}}
\|Ef_S\|^p_{BL^p_{k,A_{j+1}}(B_S)}.
\end{equation}
By Lemma \ref{l2-orthogonality} and Corollary \ref{local-l2-orthogonality-2}, for all $1\le \rho \le \rho_{j+1}$, we have
\begin{equation} \label{r3.11}
{\mathop{\mathrm{max}}_{S\in \mathcal{S}} {\left\|f_S\right\|}^2_2\ }\le C_{\mathrm{tang}}\mathop{\mathrm{max}}_{O_j\in {\mathcal{O}}_j}{\left\|f_{O_j}\right\|}^2_2
\end{equation}
\begin{equation} \label{r3.12}
{\mathop{\mathrm{max}}_{ \begin{array}{c}
S\in \mathcal{S} \\ 
\theta :{\rho }^{-1/2}\mathrm{-}\mathrm{sector} \end{array}
} {\left\|f_S\right\|}^2_{L^2_{\mathrm{avg}}\left(\theta \right)}\ }\le C_{\mathrm{tang}}{\mathop{\mathrm{max}}_{ \begin{array}{c}
O_j\in {\mathcal{O}}_j \\ 
\theta :{\rho }^{-1/2}\mathrm{-}\mathrm{sector} \end{array}
} {\left\|f_{O_j}\right\|}^2_{L^2_{\mathrm{avg}}\left(\theta \right)}\ }
\end{equation}
Since $\#{\mathcal{B}}_{O_j}\lesssim r^{n{\widetilde{\delta }}_{m-1}}$ for all $O_j\in {\mathcal{O}}_{j\mathrm{,alg}}$, Lemma \ref{l2-orthogonality} gives
\begin{equation} \label{r3.13}
\sum_{S\in \mathcal{S}}{{\left\|f_S\right\|}^2_2}\le C_{\mathrm{tang}}r^{n{\widetilde{\delta }}_{m-1}}\sum_{O_j\in {\mathcal{O}}_j}{{\left\|f_{O_j}\right\|}^2_2}
\end{equation}
Moreover, for $\mathcal{T}\subseteq \mathbb{T}[\rho_{j+1}]$ and $S=(\mathbf{Y},B)\in\mathcal{S}$ with $B\in\mathcal{B}_{O_j}$, Lemma \ref{l2-orthogonality} together with Corollary~\ref{c2.7} yields
\begin{equation} \label{r3.14}
{\left\|f_{S,\mathcal{T}}\right\|}^2_2\le C_{\mathrm{tang}}{\left\|f_{O_j,\uparrow \mathcal{T}}\right\|}^2_2
\end{equation}
where $\big(f_{S,\mathcal{T}}\big)^{\widetilde{}}:= {\left(f_S\right)}_{\mathcal{T}}^{\widetilde{}}$\\\\
\underbar{Transverse Subcase}: If \eqref{r3.9} fails, then \eqref{r3.8} gives
\[\sum_{O_j\in {\mathcal{O}}_{j\mathrm{,alg}}}{{\left\|Ef_{O_j}\right\|}^p_{BL^p_{k,A_j}\left(O_j\right)}}\lesssim \sum_{O_j\in {\mathcal{O}}_{j\mathrm{,alg}}}{\sum_{B\in {\mathcal{B}}_{O_j}}{{\left\|Ef_{B,\mathrm{t}\mathrm{rans}}\right\|}^p_{BL^p_{k,2A_{j+1}}\left(B\right)}}}\] 
Let $f_{B,\mathrm{t}\mathrm{rans}\mathrm{,ess}}:= {\left(f_{B,\mathrm{t}\mathrm{rans}}\right)}_{\mathrm{ess}}$. By \eqref{r2.12},
\[\sum_{O_j\in {\mathcal{O}}_{j\mathrm{,alg}}}{{\left\|Ef_{O_j}\right\|}^p_{BL^p_{k,A_j}\left(O_j\right)}}\lesssim \sum_{O_j\in {\mathcal{O}}_{j\mathrm{,alg}}}{\sum_{B\in {\mathcal{B}}_{O_j}}{{\left\|Ef_{B,\mathrm{t}\mathrm{rans}\mathrm{,ess}}\right\|}^p_{BL^p_{k,A_{j+1}}\left(B\right)}}}\] 
Fix $O_j\in {\mathcal{O}}_{j\mathrm{,alg}}$ and $B\in {\mathcal{B}}_{O_j}$. Arguing as in \cite{G2,OW}, we obtain a finite set $\mathfrak{b}_{B,O_j}\subseteq B(0,\rho_j^{1/2+\delta_m})$ and, for each $b\in\mathfrak{b}_{B,O_j}$, a function $g_b$ tangent to $(\mathbf{Z}_{O_j}+b,B)$ such that
\[{\left\|Ef_{B,\mathrm{t}\mathrm{rans}\mathrm{,ess}}\right\|}^p_{BL^p_{k,A_{j+1}}\left(B\right)}\lesssim {\mathrm{log} r\ }\sum_{b\in {\mathfrak{b}}_{B,O_j}}{{\left\|Eg_b\right\|}^p_{BL^p_{k,A_{j+1}}\left(B\cap N_{{\rho }^{1/2+{\delta }_m}_j}\left({\boldsymbol{\mathrm{Z}}}_{O_j}+b\right)\right)}}\] 
\[\sum_{b\in {\mathfrak{b}}_{B,O_j}}{{\left\|g_b\right\|}^2_2}\lesssim {\left\|f_{B,\mathrm{t}\mathrm{rans}\mathrm{,ess}}\right\|}^2_2\] 
Let 
\[\mathcal{O}\left(O_j\right):= \left\{B\cap N_{{\rho }^{1/2+{\delta }_m}_{j+1}}\left({\boldsymbol{\mathrm{Z}}}_{O_j}+b\right):b\in {\mathfrak{b}}_{B,O_j},\ B\in {\mathcal{B}}_{O_j}\right\}\] 
For $O=B\cap N_{{\rho }^{1/2+{\delta }_m}_{j+1}}\left({\boldsymbol{\mathrm{Z}}}_{O_j}+b\right)\in \mathcal{O}\left(O_j\right)$, define $f_O:= g_b$. Finally, set
\[{\mathcal{O}}_{j+1}:= \bigcup_{O_j\in {\mathcal{O}}_{j\mathrm{,alg}}}{\mathcal{O}\left(O_j\right)}\] 

We now verify the properties. Properties 1 and 2 follow from \cite{HR}; here we verify Properties 3 and 4.\\\\
\textbf{Property 3.} By Lemma \ref{l2.25}, 
\[{\left\|f_{O_{j+1}}\right\|}^2_2\lesssim r^{O\left({\delta }_m\right)}{\left({\rho }_j/{\rho }_{j+1}\right)}^{-\left(n-m\right)/2}{\left\|f_{O_j}\right\|}^2_2\]
Combining this with \eqref{r3.3}, we obtain
\[
\|f_{O_{j+1}}\|_2^2
\lesssim 
r^{O(\delta_m)-\overline{C}\delta_m}
d^{-\delta}
C^{\mathrm{III}}_{j+1,\delta}(d,r)
\left(\frac{r}{\rho_{j+1}}\right)^{-(n-m)/2}
d^{-\#\mathrm{c}(j+1)(m-1)}
\|f\|_2^2 .
\]
The induction closes provided $\overline{C}$ and $d$ are chosen sufficiently large. For a $\rho^{-1/2}$-sector $\theta$ with $1\le \rho \le \rho_{j+1}$,  Lemma \ref{l2.6}, Lemma \ref{l2.26}, and Corollary \ref{local-l2-orthogonality-1} yield
\[{\left\|f_{O_{j+1}}\right\|}^2_{L^2_{\mathrm{avg}}\left(\theta \right)}\lesssim r^{O\left({\delta }_m\right)}{\left({\rho }_j/{\rho }_{j+1}\right)}^{-\left(n-m\right)/2}\sum_{ \begin{array}{c}
{\theta }':{\rho }^{-1/2}\mathrm{-}\mathrm{sector} \\ 
{\theta }'\cap C\theta \neq \emptyset  \end{array}
}{{\left\|f_{O_j}\right\|}^2_{L^2_{\mathrm{avg}}\left({\theta }'\right)}}\] 
By pigeonholing, there exists a $\rho^{-1/2}$-sector $\theta'$ such that 
$\theta'\cap C\theta\neq\emptyset$ and
\[\|f_{O_{j+1}}\|^2_{L^2_{\mathrm{avg}}(\theta)}
\lesssim 
r^{O(\delta_m)}
\left(\frac{\rho_j}{\rho_{j+1}}\right)^{-(n-m)/2}
\|f_{O_j}\|^2_{L^2_{\mathrm{avg}}(\theta')}\]
Combining this with \eqref{r3.4}, we obtain
\begin{multline*}
\max_{\theta:\,\rho^{-1/2}\text{-sector}}
\|f_{O_{j+1}}\|^2_{L^2_{\mathrm{avg}}(\theta)}
\lesssim \\
r^{O(\delta_m)-\overline{C}\delta_m}
d^{-\delta}
C^{\mathrm{III}}_{j+1,\delta}(d,r)
\left(\frac{r}{\rho_{j+1}}\right)^{-(n-m)/2}
\max_{\theta:\,\rho^{-1/2}\text{-sector}}
\|f\|^2_{L^2_{\mathrm{avg}}(\theta)}
\end{multline*}
Since $\overline{C}$ and $d$ are chosen sufficiently large, this verifies Property 3.\\\\
\textbf{Property 4.} For $\mathcal{T}\subseteq \mathbb{T}\left[{\rho }_{j+1}\right]$, Lemma \ref{l2.26} implies
\[{\left\|{\tilde{f}}_{O_{j+1,\mathcal{T}}}\right\|}^2_2\lesssim r^{O\left({\delta }_m\right)}{\left({\rho }_j/{\rho }_{j+1}\right)}^{-\left(n-m\right)/2}{\left\|f_{O_j,\uparrow \mathcal{T}}\right\|}^2_2\]
Combining this with \eqref{r3.5} we have
\begin{align*}
\|f_{O_{j+1},\mathcal T}\|_2^2
&\lesssim r^{O(\delta_m)-\overline C\delta_m}d^{-\delta}
 C^{\mathrm{IV}}_{j+1,\delta}(d,r)\\
&\qquad\times\left(\frac r{\rho_{j+1}}\right)^{-(n-m)/2}
 \|f_{\uparrow^j(\uparrow\mathcal T)}\|_2^2.
\end{align*}
Since $\overline{C}$ and $d$ are chosen sufficiently large, one has
\[{\left\|f_{O_{j+1},\ \mathcal{T}}\right\|}^2_2\le C^{\mathrm{III}}_{j+1,\delta }\left(d,r\right){\left(r/{\rho }_{j+1}\right)}^{-\left(n-m\right)/2}{\left\|f_{{\uparrow }^j\left(\uparrow \mathcal{T}\right)}\right\|}^2_2\] 
As in the cellular case, we have $\uparrow^j(\uparrow \mathcal{T}) \subseteq \uparrow^{j+1}\mathcal{T}$, and thus the induction closes.\\\\
\textbf{Summary.} We now summarize the output in a concise form. Since
\begin{equation} \label{r3.18}
\#\mathrm{c}\left(J\right)\lesssim \log{r},\ \ \#\mathrm{a}\left(J\right)\lesssim {\delta }^{-1}_{m-1} \log{{\delta }^{-1}_{m-1}}
\end{equation}
the accumulated error remains negligible after the algorithm terminates. Let $D:=d^{\#\mathrm{c}(J)}$. If the \textup{[tang]} case occurs, then \eqref{r3.1}, \eqref{r3.2}, \eqref{r3.3}, \eqref{r3.4}, \eqref{r3.5}, \eqref{r3.10}, \eqref{r3.11}, \eqref{r3.12}, \eqref{r3.13}, and \eqref{r3.14} give
\begin{equation} \label{r3.19}
{\left\|Ef\right\|}^p_{BL^p_{k,A}\left(B_r\right)}\lessapprox \sum_{S\in \mathcal{S}}{{\left\|Ef_S\right\|}^p_{BL^p_{k,A_J}\left(B_S\right)}}
\end{equation}
\begin{equation} \label{r3.20}
\sum_{S\in \mathcal{S}}{{\left\|f_S\right\|}^2_2}\lessapprox D{\left\|f\right\|}^2_2
\end{equation}
\begin{equation} \label{r3.21}
{\mathop{\mathrm{max}}_{S\in \mathcal{S}} {\left\|f_S\right\|}^2_2\ }\lessapprox D^{-(m-1)}{\left(r/{\rho }_J\right)}^{-\left(n-m\right)/2}{\left\|f\right\|}^2_2
\end{equation}
\begin{equation} \label{r3.22}
{\mathop{\mathrm{max}}_{\theta :{\rho }^{-1/2}\mathrm{-}\mathrm{sector}} {\left\|f_S\right\|}^2_{L^2_{\mathrm{avg}}\left(\theta \right)}\ }\lessapprox{\left(r/{\rho }_J\right)}^{-\left(n-m\right)/2}{\mathop{\mathrm{max}}_{\theta :{\rho }^{-1/2}\mathrm{-}\mathrm{sector}} {\left\|f\right\|}^2_{L^2_{\mathrm{avg}}\left(\theta \right)}\ }
\end{equation}
\begin{equation} \label{r3.23}
{\left\|f_{S,\ \mathcal{T}}\right\|}^2_2\lessapprox {\left(r/{\rho }_J\right)}^{-\left(n-m\right)/2}{\left\|f_{{\uparrow }^J\mathcal{T}}\right\|}^2_2
\end{equation}

If the \textup{[tiny]} case occurs, then \eqref{r3.1}, \eqref{r3.2}, \eqref{r3.3}, \eqref{r3.4}, and \eqref{r3.5} give
\begin{equation} \label{r3.24}
{\left\|Ef\right\|}^p_{BL^p_{k,A}\left(B_r\right)}\lessapprox \sum_{O\in \mathcal{O}}{{\left\|Ef_O\right\|}^p_{BL^p_{k,A_J}\left(O\right)}}
\end{equation}
\begin{equation} \label{r3.25}
\sum_{O\in \mathcal{O}}{{\left\|f_O\right\|}^2_2}\lessapprox D{\left\|f\right\|}^2_2
\end{equation}
\begin{equation} \label{r3.26}
{\mathop{\mathrm{max}}_{O\in \mathcal{O}} {\left\|f_O\right\|}^2_2\ }\lessapprox D^{-(m-1)}{\left(r/{\rho }_J\right)}^{-\left(n-m\right)/2}{\left\|f\right\|}^2_2
\end{equation}

\subsection{Algorithm 2}\label{induction-on-dimension}
Algorithm 2 is obtained by repeatedly applying Algorithm 1 and terminates when the \textup{[tiny]} case dominates. Consider a family of Lebesgue exponents $p_i$ for $k\le i\le n$, which will be determined in Section \ref{conclude the proof}, satisfying
\begin{equation}\label{define_p}
    p_k\ge p_{k+1}\ge \dots \ge p_n=:p>2
\end{equation}
and define $0\le {\alpha }_i,{\beta }_i\le 1$ by ${\alpha }_n={\beta }_n:= 1$ and
\begin{equation} \label{r3.27} {\alpha }_i:= \left(\frac{1}{2}-\frac{1}{p_{i+1}}\right){\left(\frac{1}{2}-\frac{1}{p_i}\right)}^{-1},\ \ {\beta }_i:= \left(\frac{1}{2}-\frac{1}{p}\right){\left(\frac{1}{2}-\frac{1}{p_i}\right)}^{-1}\ \mathrm{for}\ k\le i\le n-1
\end{equation}
\underbar{Input}: A function $f:2{\overline{B}}^{n-1}\backslash B^{n-1}\to \mathbb{C}$ with
\begin{equation} \label{r3.28} {\left\|Ef\right\|}_{BL^p_{k,A}\left(B_R\right)}\gtrsim R^{\varepsilon }{\left\|f\right\|}_2
\end{equation} 
Here, $A\in \mathbb{N}$ is chosen to be sufficiently large.\\\\
\underline{Output.} The algorithm produces the following objects:
\begin{itemize}
\item A collection $\mathcal{O}$ of cells, each contained in an 
$R^{O(\delta_0)}$-ball.

\item For each $O\in\mathcal{O}$, a function $f_O$.

\item Scales $\overrightarrow{r}=(r_n,\dots,r_{m-1})$ satisfying 
$R=r_n>r_{n-1}>\cdots>r_{m-1}=1$.

\item Integers $\overrightarrow{A}=(A_n,\dots,A_{m-1})$ satisfying 
$A=A_n>A_{n-1}>\cdots>A_{m-1}$.

\item A tuple of parameters $\overrightarrow{D}=(D_n,\dots,D_{m-1})$.

\item For $m\le l\le n$, a family $\overrightarrow{\mathcal{S}}_l$ of 
level $(n-l+1)$ multigrains. Each 
$\overrightarrow{S}_l=((Z_l,B_{r_l}),\dots,(Z_n,B_{r_n}))\in 
\overrightarrow{\mathcal{S}}_l$ has multiscale 
$\overrightarrow{r}_l=(r_l,\dots,r_n)$ and bounded degree. Moreover, 
for $l\le s\le t\le n$,
\[
Z_s=Z'_s+b \quad \text{for some } b \text{ with } |b|\lesssim r_t^{1/2+\delta_t},
\qquad Z'_s\subseteq Z_t.
\]

\item For $m\le l\le n$, an assignment of a function 
$f_{\overrightarrow{S}_l}$ to each 
$\overrightarrow{S}_l\in\overrightarrow{\mathcal{S}}_l$. 
Each $f_{\overrightarrow{S}_l}$ is tangent to 
$S_l=(Z_l,B_{r_l})$, the first component of $\overrightarrow{S}_l$.
\end{itemize}
\underbar{Properties}: \\
1. The inequality
\begin{equation} \label{r3.29}
    {\left\|Ef\right\|}_{BL^p_{k,A}\left(B_R\right)}\lessapprox M\left(\overrightarrow{r},\overrightarrow{D}\right){\left\|f\right\|}^{1-{\beta }_m}_2{\left(\sum_{O\in \mathcal{O}}{{\left\|Ef_O\right\|}^{p_m}_{BL^{p_m}_{k,A_{m-1}}\left(O\right)}}\right)}^{{\beta }_m/p_m\ }
\end{equation} holds, where
\begin{equation} \label{r3.30}
    M\left(\overrightarrow{r},\overrightarrow{D}\right):= {D^{\delta }_{m-1}\left(\prod^{n-1}_{i=m}{D_i}\right)}^{\left(n-m\right)\delta }\left(\prod^{n-1}_{i=m}{r^{\left({\beta }_{i+1}-{\beta }_i\right)/2}_iD^{\left({\beta }_{i+1}-{\beta }_m\right)/2}_i}\right)
\end{equation}
2. The inequality
\begin{equation} \label{r3.31}
    \sum_{O\in \mathcal{O}}{{\left\|f_O\right\|}^2_2}\lessapprox \left(\prod^{n-1}_{i=m-1}{D_i}\right){\left\|f\right\|}^2_2
\end{equation}
holds.\\
3. For $m\le l\le n$,
\begin{equation} \label{r3.32}
{\mathop{\mathrm{max}}_{O\in \mathcal{O}} {\left\|f_O\right\|}^2_2\ }\lessapprox r^{-\frac{n-l}{2}}_l\left(\prod^{l-1}_{i=m-1}{r^{-1/2}_iD^{-i}_i}\right){\mathop{\mathrm{max}}_{{\overrightarrow{S}}_l\in {\overrightarrow{\mathcal{S}}}_l} {\left\|f_{{\overrightarrow{S}}_l}\right\|}^2_2\ }
\end{equation}
4. For $\ m\le l\le n-1$, ${\overrightarrow{S}}_l\in {\overrightarrow{\mathcal{S}}}_l$, and ${\overrightarrow{S}}_{n-1}\in {\overrightarrow{\mathcal{S}}}_{n-1}$ with ${\overrightarrow{S}}_l\preccurlyeq {\overrightarrow{S}}_{n-1}$
\begin{equation} \label{r3.33}
{\left\|f_{{\overrightarrow{S}}_l}\right\|}^2_2\lessapprox r^{\frac{n-l-1}{2}}_l\prod^{n-1}_{j=l+1}{r^{-1/2}_j}{\left\|f^{\#}_{{\overrightarrow{S}}_l}\right\|}^2_2
\end{equation}
where $f^{\#}_{{\overrightarrow{S}}_l}:= {\left(f_{{\overrightarrow{S}}_{n-1}}\right)}_{\mathbb{T}\left[{\overrightarrow{S}}^{\#}_l\right]}$. \\
\underbar{Initialization}. We set
\[r_n:= R, D_n:= 1, A_n:= A\] 
\[{\overrightarrow{\mathcal{S}}}_n:= \left\{{\overrightarrow{S}}_n\right\}, {\overrightarrow{S}}_n:= \left(S_n\right), S_n=\left({\mathbb{R}}^n,\ B_R\right), f_{{\overrightarrow{S}}_n}:= f\] 
\underbar{$\left(n+2-l\right)$-th step}. Suppose the algorithm has run for $(n+1-l)$ steps and the following data have been determined.
\[{\overrightarrow{r}}_l=\left(r_n,\dots ,r_l\right),\ \ {\overrightarrow{A}}_l=\left(A_n,\dots ,A_l\right),\ \ {\overrightarrow{D}}_l=\left(D_n,\dots ,D_l\right),\ \ {\overrightarrow{\mathcal{S}}}_l\] 
For each $\overrightarrow{S}_l\in\overrightarrow{\mathcal{S}}_l$, we apply Algorithm 1 to the $k$-broad norm $\|Ef_{\overrightarrow{S}_l}\|_{BL^{p_l}_{k,A_l}(B_{r_l})}$. Let $r[\overrightarrow{S}_l]$, $D[\overrightarrow{S}_l]$, and $A[\overrightarrow{S}_l]$ denote the quantities $\rho_J$, $D$, and $A_J$ returned by Algorithm 1, respectively. Let
\[{\overrightarrow{\mathcal{S}}}_{l\mathrm{,tiny}}:= \left\{{\overrightarrow{S}}_l\in {\overrightarrow{\mathcal{S}}}_l\mathrm{:}\mathrm{[alg\ 1]\ terminates\ due\ to\ [t}\mathrm{iny}\mathrm{]}\right\}\] 
\[{\overrightarrow{\mathcal{S}}}_{l\mathrm{,tang}}:= \left\{{\overrightarrow{S}}_l\in {\overrightarrow{\mathcal{S}}}_l\mathrm{:}\mathrm{[alg\ 1]\ terminates\ due\ to\ [t}\mathrm{ang}\mathrm{]}\right\}\]

Consequently, at least one of the following two cases must hold:\\\\
\textbf{Tiny-dominated Case.} We say that the tiny-dominated case happens if
\begin{equation} \label{r3.34}
    \sum_{{\overrightarrow{S}}_l\in {\overrightarrow{\mathcal{S}}}_l}{{\left\|Ef_{{\overrightarrow{S}}_l}\right\|}^{p_l}_{BL^{p_l}_{k,A_l}\left(B_{r_l}\right)}}\le 2\sum_{{\overrightarrow{S}}_l\in {\overrightarrow{\mathcal{S}}}_{l\mathrm{,tiny}}}{{\left\|Ef_{{\overrightarrow{S}}_l}\right\|}^{p_l}_{BL^{p_l}_{k,A_l}\left(B_{r_l}\right)}}
\end{equation}
By \eqref{r3.18}, there exists ${\overrightarrow{\mathcal{S}}}'_{l\mathrm{,tiny}}\subseteq {\overrightarrow{\mathcal{S}}}_{l\mathrm{,tiny}}$ and $A_{l-1}\in \mathbb{N}$ such that $A\left[{\overrightarrow{S}}_l\right]=A_{l-1}$ and $D\left[{\overrightarrow{S}}_l\right]=D_{l-1}$ for all ${\overrightarrow{S}}_l\in {\overrightarrow{\mathcal{S}}}'_{l\mathrm{,tiny}}$ and
\begin{equation} \label{r3.35}
\sum_{{\overrightarrow{S}}_l\in {\overrightarrow{\mathcal{S}}}_l}{{\left\|Ef_{{\overrightarrow{S}}_l}\right\|}^{p_l}_{BL^{p_l}_{k,A_l}\left(B_{r_l}\right)}}\lessapprox \sum_{{\overrightarrow{S}}_l\in {\overrightarrow{\mathcal{S}}}'_{l\mathrm{,tiny}}}{{\left\|Ef_{{\overrightarrow{S}}_l}\right\|}^{p_l}_{BL^{p_l}_{k,A_l}\left(B_{r_l}\right)}}
\end{equation} 
For each $\overrightarrow{S}_l\in\overrightarrow{\mathcal{S}}'_{l,\mathrm{tiny}}$, 
let $\mathcal{O}[\overrightarrow{S}_l]$ denote the final cells produced by 
Algorithm 1. For each $O\in\mathcal{O}[\overrightarrow{S}_l]$, let $f_O$ be the 
corresponding function returned by Algorithm 1. Algorithm 2 then outputs the following additional data and terminates:
\[
m=l, \qquad r_{m-1}:=1, \qquad 
\mathcal{O}:=\bigcup_{\overrightarrow{S}_m\in 
\overrightarrow{\mathcal{S}}'_{m,\mathrm{tiny}}}
\mathcal{O}[\overrightarrow{S}_m].
\]
\textbf{Tangent-dominated Case.} If \eqref{r3.34} fails, then
\[\sum_{{\overrightarrow{S}}_l\in {\overrightarrow{\mathcal{S}}}_l}{{\left\|Ef_{{\overrightarrow{S}}_l}\right\|}^{p_l}_{BL^{p_l}_{k,A_l}\left(B_{r_l}\right)}}\le 2\sum_{{\overrightarrow{S}}_l\in {\overrightarrow{\mathcal{S}}}_{l\mathrm{,tang}}}{{\left\|Ef_{{\overrightarrow{S}}_l}\right\|}^{p_l}_{BL^{p_l}_{k,A_l}\left(B_{r_l}\right)}}\]
By \eqref{r3.18}, there exists ${\overrightarrow{\mathcal{S}}}'_{l\mathrm{,tang}}\subseteq {\overrightarrow{\mathcal{S}}}_{l\mathrm{,tang}}$, $A_{l-1}\in \mathbb{N}$, and $r_{l-1}\in[1,r_l)$ such that
\begin{equation} \label{r3.37}
\sum_{{\overrightarrow{S}}_l\in {\overrightarrow{\mathcal{S}}}_l}{{\left\|Ef_{{\overrightarrow{S}}_l}\right\|}^{p_l}_{BL^{p_l}_{k,A_l}\left(B_{r_l}\right)}}\lessapprox \sum_{{\overrightarrow{S}}_l\in {\overrightarrow{\mathcal{S}}}'_{l\mathrm{,tang}}}{{\left\|Ef_{{\overrightarrow{S}}_l}\right\|}^{p_l}_{BL^{p_l}_{k,A_l}\left(B_{r_l}\right)}}
\end{equation} and
\[A\left[{\overrightarrow{S}}_l\right]=A_{l-1},\ \ r\left[{\overrightarrow{S}}_l\right]=r_{l-1},\ \ D\left[{\overrightarrow{S}}_l\right]=D_{l-1}, \ \ \forall {\overrightarrow{S}}_l\in {\overrightarrow{\mathcal{S}}}'_{l\mathrm{,tang}}\]
For each $\overrightarrow{S}_l\in\overrightarrow{\mathcal{S}}'_{l,\mathrm{tang}}$, let $S[\overrightarrow{S}_l]$ denote the collection of grains produced by Algorithm 1. For each $S_{l-1}\in S[\overrightarrow{S}_l]$, define
\[
\overrightarrow{S}_{l-1}:=(S_{l-1},\overrightarrow{S}_l),
\]
and let $f_{\overrightarrow{S}_{l-1}}$ denote the corresponding function returned by Algorithm 1. Finally, we define
\[{\overrightarrow{\mathcal{S}}}_{l\mathrm{-}\mathrm{1}}:= \bigcup_{{\overrightarrow{S}}_l\in {\overrightarrow{\mathcal{S}}}'_{l\mathrm{,tang}}}{\left\{{\overrightarrow{S}}_{l-1}:S_{l\mathrm{-}\mathrm{1}}\in S\left[{\overrightarrow{S}}_l\right]\right\}}\]\\
\textbf{Verification of the Properties.} Property 2 follows from \eqref{r3.25} together with repeated applications of \eqref{r3.20}. Property 3 follows from \eqref{r3.26} and repeated applications of \eqref{r3.21}. Property 1 follows from \eqref{r3.19}, \eqref{r3.37}, \eqref{r3.24}, \eqref{r3.35}, Lemmas \ref{l2.18} and \ref{l2.19}; see \cite{HR} for details. We now verify Property 4.

Applying \eqref{r3.23} repeatedly yields
\[
\|f_{\overrightarrow{S}_l}\|_2^2
\lessapprox 
\prod_{i=l}^{n-2}
\left(\frac{r_{i+1}}{r_i}\right)^{-(n-i-1)/2}
\|f^{\#}_{\overrightarrow{S}_l}\|_2^2
=
r_l^{(n-l-1)/2}
\prod_{j=l+1}^{n-1}
r_j^{-1/2}\|f^{\#}_{\overrightarrow{S}_l}\|_2^2 .\]
This establishes Property 4.
\subsection{Establishing the broad estimate}\label{k-broad estimates}
We now apply the properties provided by Algorithm~2. Since each $O\in\mathcal{O}$ has diameter at most $R^{O(\delta_0)}$, we may trivially bound
\[{\left\|Ef_O\right\|}_{BL^{p_m}_{k,A_{m-1}}\left(O\right)}\lessapprox {\left\|f_O\right\|}_2\]
Together with \eqref{r3.29}, this yields
\begin{equation} \label{r3.39}
    {\left\|Ef\right\|}_{BL^p_{k,A}\left(B_R\right)}\lessapprox M\left(\overrightarrow{r},\overrightarrow{D}\right){\left\|f\right\|}^{1-{\beta }_m}_2{\left(\sum_{O\in \mathcal{O}}{{\left\|f_O\right\|}^2_2}\right)}^{{\beta }_m/p_m\ }{\mathop{\mathrm{max}}_{O\in \mathcal{O}} {\left\|f_O\right\|}^{2{\beta }_m\left(\frac{1}{2}-\frac{1}{p_m}\right)}_2\ }
\end{equation}
Combining this with \eqref{r3.30} and \eqref{r3.27}, we obtain
We set $\beta_{m-1}:=0$; the corresponding scale is the dummy scale $r_{m-1}=1$. Tracking the degree factors in \eqref{r3.30} and the preceding applications of Algorithm~1, there is a constant $C=C(n)$ such that
\begin{equation} \label{r3.40}
{\left\|Ef\right\|}_{BL^p_{k,A}\left(B_R\right)}\lessapprox R^{C\delta_0}\prod^{n-1}_{i=m-1}{r^{\frac{{\beta }_{i+1}-{\beta }_i}{2}}_iD^{\frac{{\beta }_{i+1}}{2}-\left(\frac{1}{2}-\frac{1}{p}\right)+C\delta}_i}{\left\|f\right\|}^{\frac{2}{p}}_2{\mathop{\mathrm{max}}_{O\in \mathcal{O}} {\left\|f_O\right\|}^{1-\frac{2}{p}}_2\ }.
\end{equation}

\bigskip

\section{Completing the proof} \label{conclude the proof}
In this section, we complete the proof of Proposition~\ref{t2.22}. We first fix the auxiliary exponents used in Algorithm~2; importantly, these exponents depend on $k$ and not on the terminal dimension $m$.

Set
\[
a_j:=(k-1)\prod_{i=k}^{j-1}\frac{2i}{2i+1},\qquad k\le j\le n-1,
\]
with the convention $a_k=k-1$, and let
\[
P:=p(n,k),\qquad t:=\frac12-\frac1P.
\]
A direct calculation gives
\[
t^{-1}=\frac{4n+1+2a_{n-1}}3.
\]
For $k\le j\le n-1$, define
\begin{equation}\label{define_exponents}
 p_j^{(0)}:=2+\frac6{2(j-1)+a_j},
 \qquad p_n^{(0)}:=P,
\end{equation}
and put
\[
 b_j:=\left(\frac12-\frac1{p_j^{(0)}}\right)^{-1}
      =\frac{4j+2+2a_j}{3},\qquad
 B:=t^{-1}=b_{n-1}+1.
\]
Thus $\beta_j^{(0)}:=b_j/B=tb_j$ for $k\le j\le n-1$ and $\beta_n^{(0)}=1$. Since $a_{j+1}=2ja_j/(2j+1)$, the sequence $b_j$ is increasing. Hence
\[
p_k^{(0)}\ge p_{k+1}^{(0)}\ge\cdots\ge p_n^{(0)}=P>2,
\]
so these are admissible auxiliary exponents for Algorithm~2.

We also define
\[
 s_j:=\frac13-\frac{2a_j}{3(2j+1)},\qquad k\le j\le n-2,
 \qquad s_{k-1}:=0.
\]
The recursion for $a_j$ yields
\begin{equation}\label{fixed-k-identities}
 b_{j+1}-b_j=1+s_j,
 \qquad
 b_{j+1}=2\bigl(1+j(1-s_j)\bigr),
\end{equation}
and
\[
 s_j-s_{j-1}=\frac{k-1}{2j(j-1)}
 \prod_{i=k}^{j}\frac{2i}{2i+1}>0.
\]
Finally, set
\begin{equation}\label{define_gamma}
 g:=1-s_{n-2}=\frac23+\frac{a_{n-1}}{3(n-2)}.
\end{equation}

We now apply Algorithm~2 with a small perturbation of the preceding auxiliary exponents. Fix $0<\eta<1/2$ and define
\[
 \beta_j:=(1-\eta)\beta_j^{(0)}\quad(k\le j\le n-1),
 \qquad \beta_n:=1,
\]
and choose $p_j$ by
\[
 \frac1{p_j}:=\frac12-\frac{t}{\beta_j},\qquad k\le j\le n-1,
 \qquad p_n:=P.
\]
For $\eta$ sufficiently small these exponents preserve the required ordering. The perturbation is used only to absorb the factors $D_i^{O(\delta)}$ in \eqref{r3.40}; the root exponent $P$ is unchanged.

Apply Algorithm~2 to $\|Ef\|_{BL^P_{k,A}(B_R)}$. By Lemma~\ref{l2.20}, the terminal dimension satisfies $m\ge k$.

\phantomsection\label{define_exponents_2}\label{r3.42}
\textbf{The cases $m=n-1,n$.}
Taking $l=n$ in \eqref{r3.32}, we obtain
\begin{equation} \label{r3.41}
\max_{O\in\mathcal O}\|f_O\|_2^2
\lessapprox
\left(\prod_{j=m-1}^{n-1}r_j^{-1/2}D_j^{-j}\right)\|f\|_2^2.
\end{equation}
For $m=n-1$, the only nontrivial scale exponent arising from \eqref{r3.40} and \eqref{r3.41} is at most $\eta\beta_{n-1}^{(0)}/2$, while the principal degree exponents are
\[
\frac{\beta_{n-1}^{(0)}}2-(n-1)t
=\frac t3\bigl(a_{n-1}-(n-2)\bigr)\le0,
\qquad
\frac12-nt<0.
\]
For $m=n$ only the second degree exponent occurs and $r_{n-1}=1$. After the perturbation, all nontrivial degree exponents are strictly negative, while the possible positive scale loss is at most $R^{\eta/2}$. Since the frequency support has bounded measure,
\[
 \|f\|_2\lesssim \|f\|_2^{2/Q}\|f\|_\infty^{1-2/Q}
\]
for every $2\le Q\le P$. Thus these two cases satisfy the desired estimate once $\eta$ and then $\delta$ are chosen sufficiently small relative to $\varepsilon$.

\textbf{The case $k\le m\le n-2$.}
We first use the nested polynomial Wolff axiom.
\begin{lemma} \label{l4.1} For $m\le l \le n-2$,
\[
 \|f^{\#}_{\overrightarrow S_l}\|_2^2
 \lessapprox \left(\prod_{i=l}^{n-2}r_i^{-1/2}\right)\|f\|_\infty^2.
\]
\end{lemma}
\begin{proof}
Since $f^{\#}_{\overrightarrow S_l}$ is concentrated on $r_{n-1}$-scale wave packets whose sectors belong to $\Theta[\overrightarrow S_l^{\#}]$, Corollary~\ref{l2.16} gives
\begin{multline}\label{r4.1}
 \|f^{\#}_{\overrightarrow S_l}\|_2^2
 \lesssim \#\Theta[\overrightarrow S_l^{\#}]r_{n-1}^{-(n-2)/2}
 \max_{\theta:r_{n-1}^{-1/2}\text{-sector}}
 \|f^{\#}_{\overrightarrow S_l}\|_{L^2_{\rm avg}(\theta)}^2\\
 \lessapprox \left(\prod_{j=l}^{n-2}r_j^{-1/2}\right)
 \max_{\theta:r_{n-1}^{-1/2}\text{-sector}}
 \|f^{\#}_{\overrightarrow S_l}\|_{L^2_{\rm avg}(\theta)}^2.
\end{multline}
By Corollary~\ref{local-l2-orthogonality-2} and \eqref{r3.22},
\begin{equation}\label{r4.2}
 \max_{\theta:r_{n-1}^{-1/2}\text{-sector}}
 \|f^{\#}_{\overrightarrow S_l}\|_{L^2_{\rm avg}(\theta)}^2
 \lessapprox
 \max_{\theta:r_{n-1}^{-1/2}\text{-sector}}
 \|f\|_{L^2_{\rm avg}(\theta)}^2
 \lesssim \|f\|_\infty^2.
\end{equation}
Combining \eqref{r4.1} and \eqref{r4.2} proves the lemma.
\end{proof}

Combining \eqref{r3.32}, \eqref{r3.33}, and Lemma~\ref{l4.1}, for every $m\le l\le n-2$ we obtain
\begin{equation}\label{r4.3}
 \max_{O\in\mathcal O}\|f_O\|_2^2
 \lessapprox
 \left(\prod_{j=m}^{n-1}r_j^{-1/2}\right)
 \left(\prod_{j=l}^{n-2}r_j^{-1/2}\right)
 \left(\prod_{j=m-1}^{l-1}D_j^{-j}\right)\|f\|_\infty^2.
\end{equation}
For the actual terminal dimension $m$, define weights by
\[
 \gamma_m:=s_m,
 \qquad
 \gamma_j:=s_j-s_{j-1}\quad(m<j\le n-2),
 \qquad
 \gamma_{n-1}:=g.
\]
They are nonnegative and sum to one. Moreover, if
$\sigma_j:=\sum_{l=m}^j\gamma_l$, then
\[
 \sigma_j=s_j\quad(m\le j\le n-2),
 \qquad \sigma_{m-1}:=0.
\]
Taking the geometric mean of \eqref{r4.3} with weights $\gamma_l$ and \eqref{r3.41} with weight $\gamma_{n-1}$ gives
\begin{multline}\label{r4.4}
 \max_{O\in\mathcal O}\|f_O\|_2^2
 \lessapprox r_{n-1}^{-1/2}D_{n-1}^{-(n-1)g}
 \prod_{j=m}^{n-2}r_j^{-(1+s_j)/2}
 \prod_{j=m-1}^{n-2}D_j^{-j(1-\sigma_j)}\\
 \times \|f\|_2^{2g}\|f\|_\infty^{2(1-g)}.
\end{multline}
Insert \eqref{r4.4} into \eqref{r3.40}. In the unperturbed case $\eta=0$, the identities \eqref{fixed-k-identities} give, for $m\le j\le n-2$,
\[
 \frac{\beta_{j+1}^{(0)}-\beta_j^{(0)}}2
 -\frac{1+s_j}{2}t=0,
 \qquad
 \frac{\beta_{j+1}^{(0)}}2-\bigl(1+j(1-s_j)\bigr)t=0.
\]
At the bottom and top degree parameters we have
\[
 \frac{\beta_m^{(0)}}2-mt=\frac t3(a_m-(m-1))\le0,
\]
and
\[
 \frac12-\bigl(1+(n-1)g\bigr)t
 =-t\left(\frac16+\frac{a_{n-1}}{3(n-2)}\right)<0.
\]
The top scale exponent also vanishes because $B=b_{n-1}+1$.

For $\eta>0$, all the degree exponents that vanished above decrease by a fixed positive multiple of $\eta$, whereas the only new positive scale exponent is bounded by $\eta/2$. Hence the factors $D_i^{C\delta}$ and $R^{C\delta_0}$ in \eqref{r3.40} are absorbed by first choosing $\eta$ sufficiently small relative to $\varepsilon$, and then choosing $\delta$ and $\delta_0$ sufficiently small relative to $\eta$. We obtain
\[
 \|Ef\|_{BL^P_{k,A}(B_R)}
 \lesssim_\varepsilon R^\varepsilon
 \|f\|_2^{2/P+2gt}\|f\|_\infty^{1-2/P-2gt}.
\]
Define $Q$ by
\[
 \frac1Q:=\frac1P+gt.
\]
Since
\[
 g=\frac23+\frac{a_{n-1}}{3(n-2)},
 \qquad
 t=\frac3{4n+1+2a_{n-1}},
\]
a direct simplification yields
\[
 \frac1Q=\frac1P+\frac{1}{4n-8}\left(\frac6P-1\right).
\]
Thus $Q=q(n,k)$, and the preceding estimate is exactly \eqref{boardform}. This proves Proposition~\ref{t2.22}.

\end{document}